\documentclass[a4paper]{amsart}
\usepackage[utf8]{inputenc}
\usepackage[T1]{fontenc}

\usepackage{amssymb}
\usepackage{amsfonts}
\usepackage{amsthm}
\usepackage{calrsfs}
\usepackage{array}
\usepackage{bbm}
\usepackage{stmaryrd}
\usepackage{hyperref}

\usepackage{mathabx}

\usepackage{enumitem}

\usepackage{csquotes}

\usepackage{mathtools}
\mathtoolsset{showonlyrefs=true}

\newcommand{\R}{\mathbb{R}}

\newcommand{\e}{\epsilon}

\newcommand{\complex}{\mathbb{C}}

\newcommand{\jap}[1]{\langle #1 \rangle}
\newcommand{\weak}{\rightharpoonup}

\renewcommand{\Re}{\mathop{\text{Re}}}
\renewcommand{\Im}{\mathop{\text{Im}}}

\newcommand{\m}[1]{
	\ifdefequal{#1}{1}
	{\mathbbm{#1}}
	{\mathbb{#1}}
}

\newcommand{\q}[1]{\mathcal{#1}}

\newcommand{\nz}[1]{\vvvert #1 \vvvert}
\newcommand{\ds}{\displaystyle}

\DeclareMathOperator{\sgn}{\mathrm{sgn}}
\DeclareMathOperator{\Ai}{\mathrm{Ai}}

\theoremstyle{plain}
\newtheorem{thm}{Theorem}
\newtheorem*{thm*}{Theorem}
\newtheorem{prop}[thm]{Proposition}

\newtheorem{lem}[thm]{Lemma}

\theoremstyle{definition}

\theoremstyle{remark}
\newtheorem{nb}[thm]{Remark}

\numberwithin{equation}{section}
\newtheoremstyle{mytheoremstyle} 
{\topsep}                    
{\topsep}                    
{}                   
{}                           
{\scshape}                   
{.}                          
{.5em}                       
{}  

\theoremstyle{mytheoremstyle} 
\theoremstyle{mytheoremstyle} 

\makeatletter
\@namedef{subjclassname@2020}{%
  \textup{2020} Mathematics Subject Classification}
\makeatother

\date{}
\author{Sim\~ao Correia \and Raphaël Côte}

\title{Perturbation  at blow-up time of self-similar solutions for the modified Korteweg-de Vries equation}

\subjclass[2020]{35Q53, 35B44, 35C06, 35C20} %
\keywords{Self-similar solution, blow-up, modified Korteweg-de Vries equation}

\allowdisplaybreaks

\begin{document}

\setlength{\parindent}{0em}
\begin{abstract}
	\noindent We prove a first stability result of self-similar blow-up for the modified KdV equation on the line. More precisely, given a self-similar solution and a sufficiently small regular profile, there is a unique global solution which behaves at $t=0$ as the sum of the self-similar solution and the smooth perturbation. 
\end{abstract}

\maketitle

\section{Introduction}

\subsection{Description of the problem and motivation}

We consider the modified Korteweg-de Vries equation on the whole line
\begin{equation}\tag{mKdV}\label{mKdV}
\partial_t u + \partial_{xxx} u = \pm \partial_x (u^3), \quad (t,x) \in \m R^2, \quad u(t,x) \in \m R.
\end{equation}
Throughout this work, the specific sign of the nonlinearity is irrelevant. To simplify the exposition, we treat the focusing case (with the + sign), even though the results presented also hold for the defocusing one.
\bigskip

This equation admits a scaling invariance: if $u$ is a solution, so is $u_\lambda(x,t)=\lambda u(\lambda^3 t, \lambda x)$, for any $\lambda>0$. As a consequence, one may look for self-similar solutions of \eqref{mKdV}, which are invariant under scalings. A simple computation shows that these solutions are of the form
\begin{gather} \label{eq:selfsim}
S(t,x)=\frac{1}{t^{1/3}}V\left(\frac{x}{t^{1/3}}\right),\quad V'' - \frac{y}{3}V = V^3 + \alpha,\ \alpha\in\R.
\end{gather}
The existence of profiles $V$ can be studied using either ODE techniques (\cite{DZ95, DK21, HM80,  PV07}) or stationary phase arguments (\cite{CCV19}). Very precise asymptotics were obtained in both physical and frequency space. Generally speaking, self-similar profiles have the same behavior as the Airy function (which solves the linear equation), up to some logarithmic corrections. In  physical space, the profiles have weak decay and strong oscillations as $x\to -\infty$.  On Fourier side, a jump discontinuity at the zero frequency appears for $\alpha\neq 0$ and no decay is available for large frequencies (see Proposition \ref{prop:selfsimilar} for a precise description).

\bigskip

As it turns out, self-similar solutions determine the behavior of small solutions for large times. This was first seen by Deift and Zhou  in \cite{DZ93} using inverse scattering techniques, under strong smoothness and decay assumptions. In \cite{HN99,HN01}, the phenomena was proven as a consequence of modified scattering. This was later revisited in \cite{GPR16} and \cite{HaGr16}. On the other hand, self-similarity induces a natural blow-up behavior at $t=0$. This singularity is directly connected to some geometric flows. Indeed, \eqref{mKdV} appears in the modeling of the evolution of the boundary of a vortex patch on the plane subject to Euler's equations (\cite{GP92}) and in the study of vortex filaments in $\R^3$. In these models, self-similar blow-up is connected to the formation of logarithmic spirals (if $\alpha=0$, one observes a sharp corner).

\bigskip

For geometric flows modeled by the cubic nonlinear Schrödinger equation, we advise the reader to look at the series of papers \cite{ BV12, BV13,BV18} and references therein.  Both the cubic (NLS) and the (mKdV) equations are $L^1$-critical. This feature translates a critical polynomial behavior of the nonlinearity at $t=0$. In the (NLS) case, using the pseudo-conformal transformation, one can reduce the self-similar blow-up analysis at $t=0$ to a problem at $t=+\infty$. Furthermore, self-similar solutions are transformed to constants, which is of course a nice simplification. However, for the \eqref{mKdV} equation, no such transformation exists. One must handle the critical behavior and truly understand what happens at the blow-up time.

\bigskip

In a previous work with Luis Vega \cite{CCV20}, we built a critical space on which self-similar solutions naturally exist and proved local and global well-posedness for strictly positive times. This is actually a delicate issue: on one hand, the rough properties of self-similar profiles imply very mild conditions on the functional space. On the other hand,  the loss of derivative in the nonlinear term is very difficult to handle at low regularity (in the context of $H^s$ spaces, one cannot go lower than $s=1/4$). These ingredients had to be carefully balanced in order to achieve a suitable framework on which we could analyze self-similar solutions for large times. This framework will once again play a major role in the analysis at the blow-up time, as we will see later on.

\bigskip

The goal of this work is to give a first step in understanding the \eqref{mKdV} flow near self-similar solutions at time $t=0$. There are two intertwined stability problems which one may consider. The first is to start with a perturbed self-similar solution at time $t=1$ and to study the behavior as $t\to0$. The second, on which we focus here, is to construct a solution $u$ of (mKdV), defined on a small time interval around $t=0$, and such that, given a perturbation $z$,
\[
u(t)-S(t) \to z \quad \text{as } t \to 0 \quad\text{ in some appropriate norm,}
\]
We shall prove that it is possible to construct a such a solution $u$, for a large (open) class of perturbations $z$, thus showing a first result on the stability of self-similar blow-up for \eqref{mKdV}.

\bigskip

This is in the same spirit as \cite{BW98} for the $L^2$-critical (NLS), and \cite{GV13} for the cubic 1D (NLS). However, \eqref{mKdV} self-similar solutions are localized neither in physical nor Fourier space, as opposed to solitons (as in \cite{BW98}), or constant solutions (as in \cite{GV13}). Even further, the $L^1$-criticality of the equation leads to modified scattering, involving logarithmic spirals (see \cite{PV07}). These critical features in both space and time create substantial obstacles in the analysis of the linearized problem around self-similar solutions. To our knowledge, our result is the first to directly construct solutions under such a rough background.

\subsection{Definitions and statement of the main result}\label{sec:statements}

Given a function $v:I\subset \R\to \mathcal{S}'(\R)$, we define the profile 
\begin{align} \label{def:profile}
\tilde{v}(t,\xi ):=e^{it\xi^3}\hat{v}(t,\xi)
\end{align}
(we denote by $\hat\cdot$ or $\q F$ the Fourier transform in the space variable). Observe that, if $v$ is a solution of the Airy equation, then $\tilde{v}$ is constant in time. On the other hand, a self-similar solution $S$ will satisfy (with a slight abuse of notation) $\tilde{S}(t,\xi)=\tilde{S}(t^{1/3}\xi)$.

By canceling the linear evolution, the oscillatory behavior in frequency is completely  concentrated on the nonlinear term: the equation \eqref{mKdV} writes for the profile $\tilde u$
\begin{gather} \label{mkdv:profile}
\partial_t \tilde{u} =  N[\tilde u](t), \quad \text{ where} \\
 N[\tilde u](t) := \frac{i\xi}{4\pi^2}\iint_{\xi_1+\xi_2+\xi_3=\xi} e^{it(\xi^3-\xi_1^3-\xi_2^3-\xi_3^3)}\tilde{u}(t,\xi_1)\tilde{u}(t,\xi_2)\tilde{u}(t,\xi_3)d\xi_1d\xi_2.
\end{gather}
One may use stationary phase arguments (pointwise in time) to extract the main contributions of the nonlinear term. To bound properly the remainder in such an expansion, we define
\[
\|u\|_{\mathcal{E}(t)} := \|\tilde{u}(t)\|_{L^\infty} + t^{-1/6}\|\partial_\xi \tilde{u}(t)\|_{L^2(\R\setminus\{0\})}
\]
and, for any interval $I\subset (0,+\infty)$,
\[
\mathcal{E}(I)=\left\{ u:I\to \mathcal{S}'(\R): \tilde{u}\in \q C(I, L^\infty(\R)),\ \partial_\xi\tilde{u}\in L^\infty(I,L^2(\R\setminus\{0\}))  \right\}
\]
endowed with the norm

\[
\|u\|_{\mathcal{E}(I)}=\sup_{t\in I} \|u(t)\|_{\mathcal{E}(t)}.
\]
\begin{nb}
 By $L^2(\R\setminus\{0\})$, we mean the set of distributions  whose restriction to $\m R\setminus\{0\}$ identifies with an $L^2$ function.
From Sobolev's embedding, a function in $\mathcal{E}$ is $1/2$-Hölder continuous in frequency, with the possible exception of $\xi=0$, where a jump discontinuity may occur. One needs to allow this behavior in order to include self-similar solutions with $\alpha\neq 0$. As one can see in the following proofs, this jump will not introduce any extra difficulty (observe that the zero frequency is preserved by the \eqref{mKdV} flow).
\end{nb}

As it was proven in \cite{CCV20}, the space $\mathcal{E}$ is sufficient to perform the stationary phase analysis (see also \cite{HN99} for a similar development using a slightly stronger norm):

\begin{lem}[The profile equation, {\cite[Lemma 7]{CCV20}}] \label{lem:desenvolveoscilatorio}
	Let $u \in  \mathcal{E}(I)$. For all $t \in I$ and $\xi >0$,
	\begin{gather} \label{eq:expansaou}
	N[\tilde u](t,\xi) =\frac{\pi \xi^3}{\jap{\xi^3t}}\left(i|\tilde{u}(t,\xi)|^2\tilde{u}(t,\xi) - \frac{1}{\sqrt{3}}e^{-8it\xi^3/9}\tilde{u}^3\left(t,\frac{\xi}{3}\right) \right)+ R[u](t,\xi) \\
	\label{eq:N_remainder}\text{with} \qquad |R[u](t,\xi)| \lesssim \frac{\xi^3\|u(t) \|^3_{\q E(t)}}{(\xi^3t)^{5/6}\jap{\xi^3t}^{1/4}}.
	\end{gather}
	Consequently, if $u$ is a distributional solution of \eqref{mKdV} on $I$,
\begin{equation}\label{eq:estut}
\forall t \in I, \quad \| \partial_t \tilde{u}(t)\|_{L^\infty}\lesssim \frac{1}{t}\|u(t)\|_{\mathcal{E}(t)}^3.
\end{equation}
\end{lem}
One of the main observations in \cite{CCV20} is that the $\mathcal{E}$ norm is enough to bound both the nonlinear term and self-similar solutions:

\begin{prop}[Existence of self-similar solutions, {\cite[Theorem 1]{CCV19}}] \label{prop:selfsimilar}
	Given $c,\alpha\in\R$ sufficiently small, there exists a unique self-similar solution $S\in \mathcal{E}((0,+\infty))$ with
	\begin{gather*}
	\|S(t)\|_{\mathcal{E}(t)}=\|S(1)\|_{\mathcal{E}(1)}\lesssim c^2+\alpha^2 \qquad \text{for all } t >0, and\\
	S(t) \weak c\delta_{x=0}+ \alpha \mathop{\mathrm{p.v.}} \left( \frac{1}{x} \right) \quad \text{in } \mathcal{D}'(\R) \text{ as } t \to 0.
	\end{gather*}
	Furthermore, there exist $A, B\in \complex$ such that
	\begin{gather*}
\tilde{S}(t^{1/3}\xi) \sim \begin{cases}
	 A e^{i a \ln |t^{1/3}\xi|} + B \frac{e^{ia \ln |t\xi^3| -i \frac{8}{9} t\xi^3}}{t\xi^3},  &|t\xi^3|\gg1,\\c + \frac{3 i \alpha}{2\pi} \sgn(t\xi^3), & |t\xi^3|\ll1.
	\end{cases}
		\end{gather*}

\end{prop}

\begin{nb}
	The norm of self-similar solutions is preserved due to the norm $\|\cdot\|_{\mathcal{E}(t)}$ being scale-invariant.
\end{nb}
%

 Our goal in this paper is to construct solutions $u$ to \eqref{mKdV} which blow up at time $0$ as the sum of such a self-similar solution $S$ and a prescribed (more regular) perturbation.

Let us outline the scheme to derive a precise statement and its proof.
Using estimate \eqref{eq:estut}, one can hope to bootstrap the $L^\infty$ norm of $\tilde{u}$. In order to control the $\mathcal{E}$ norm, we need another key ingredient: the scaling operator $Iu$, formally defined as 
\[
Iu := xu+3t\int_{-\infty}^x \partial_tu dx',
\]
or equivalently in Fourier variable:
\begin{equation} \label{def:I}
\widehat{Iu}(t,\xi) := i \partial_\xi \hat u - \frac{3it}{\xi} \partial_t \hat u.
\end{equation}
As it can be seen in Fourier variables, the $L^2$ norm of $Iu$ is intimately related to that of  $\partial_\xi \tilde{u}$.
A direct computation yields
\begin{equation}
(\partial_t + \partial_x^3)Iu = 3u^2(Iu)_x
\end{equation}
and thus
\begin{equation}\label{eq:Iu}
\frac{d}{dt}\|Iu(t)\|_{L^2}^2\lesssim \left|\int u^2(Iu)_xIu dx\right|\lesssim \left|\int uu_x(Iu)^2 dx\right|\lesssim \frac{\|u\|_{\mathcal{E}(t)}^2}{t} \|Iu(t)\|_{L^2}^2.
\end{equation}
(recall \cite[Lemma 6]{CCV20}).
As it is clear from \eqref{eq:estut} and \eqref{eq:Iu}, the problem is marginally singular at $t=0$. This should not come as a surprise, due to the $L^1$-critical nature of the \eqref{mKdV} equation. For positive times away from $t=0$, these estimates are sufficient to construct a solution over the space $\mathcal{E}$ (see \cite{CCV20}). To explain how to  improve the behavior at $t=0$, let us look closely to \eqref{eq:expansaou} and forget the $R$ term. If, for some reason, one had $|\tilde{u}(t,\xi)|\lesssim \jap{\xi}^{-\epsilon}$ for some $\epsilon >0$, then
\[
|\partial_t \tilde{u} (t,\xi)|\lesssim \frac{\xi^{3-\epsilon}}{\jap{\xi^3t}}\sup_t\left\{\|\tilde{u}(t)\|_{L^\infty}^2\|\jap{\cdot}^\epsilon\tilde{u}(t)\|_{L^\infty}\right\} \lesssim \frac{1}{t^{1-\epsilon/3}},
\]
which can now be integrated in $(0,t)$ to produce an $L^\infty$ bound on $\tilde{u}$. There are two problems with this approach: first, as one may expect, self-similar solutions do not enjoy any extra decay in $\xi$; second, an \textit{a priori} bound for the extra decay would have to go through the profile equation, where finds once again the $1/t$ behavior at $t=0$. On the other hand, if one had $|\tilde{u}(t,\xi)|\lesssim t^{\epsilon}$, then
\[
|\partial_t \tilde{u} (t,\xi)|\lesssim \frac{\xi^{3}t^{\epsilon}}{\jap{\xi^3t}}\sup_t\left\{t^{-\epsilon}\|\tilde{u}(t)\|_{L^\infty}^3\right\} \lesssim \frac{1}{t^{1-\epsilon}},
\]
and the integration becomes possible on $(0,t)$. Unfortunately, this assumption is even more problematic, since it implies that $\tilde{u}(0,\xi)\equiv0$. It becomes clear that an extra decay in either frequency or time would suffice to derive an $L^\infty_\xi$ bound. The key idea is to decompose $u$ as 
\begin{align} \label{eq:decomp_u_Szw}
\tilde{u}(t,\xi)=\tilde{S}(t,\xi) + \tilde{z}(t,\xi) + \tilde{w}(t,\xi),
\end{align}
where $z$ has extra smoothness and we aim at bootstrapping information on $\|w\|_{\q E}$.

The self-similar solution, despite its singular behavior, is an exact solution with precise asymptotics in both space and frequency. The regular term $\tilde{z}$ can be chosen sufficiently smooth in space and frequency: in fact, as no polynomial bound in time is necessary, we will assume $\tilde{z}$ constant in time (that is, it corresponds to the linear evolution of the perturbation). The remainder term $\tilde{w}$ will satisfy a bound $\|w(t)\|_{\mathcal{E}(t)}\lesssim t^\epsilon$ and it will measure the interaction between the self-similar solution and the localized linear solution. The equation for the remainder $w$ is
\begin{equation}\label{eq:w}
\partial_t w +\partial_{xxx} w = \partial_x (u^3 - S^3), \quad w(0)=0.
\end{equation}
Observe that, since the evolution of the regular part $z$ is linear, no \textit{a priori} decay and smoothness estimates are necessary. The problem is completely reduced to the existence of $w$ over $\mathcal{E}$ with a polynomial bound in time. From the above discussion, the $L^\infty$ bound on $\tilde{w}$ should hold and we are left with the \textit{a priori} bound on $Iw$, for which the equation is
\[
(\partial_t + \partial_x^3)Iw = 3\left(u^2(Iu)_x-S^2(IS)_x\right).
\]
It is at this point that another decisive feature is revealed: due to the self-similar nature of $S$, $(IS)_x \equiv 0$. Thus
\[
(\partial_t + \partial_x^3)Iw = 3u^2(Ie^{-t\partial_x^3}z)_x + 3u^2(Iw)_x
\]
A direct integration yields
\begin{align*}
\frac{d}{dt}\|Iw\|_{L^2}^2 & \lesssim \left|\int  u^2(Ie^{-t\partial_x^3}z)_xIw dx \right| + \left|\int  u^2(Iw)_xIw dx\right| \\
& \lesssim \frac{\|u\|_{\mathcal{E}(t)}^2}{t^{2/3}}\|(Ie^{-t\partial_x^3}z)_x\|_{L^2}\|Iw\|_{L^2} + \frac{\|u\|_{\mathcal{E}(t)}^2}{t}\|Iw\|_{L^2}^2.
\end{align*}
Since $\q F (Ie^{-t\partial_x^3}z)_x = \partial_\xi \hat z$, the factor $\|(Ie^{-t\partial_x^3}z)_x\|_{L^2}$ causes no further singular behavior at $t=0$. As $Iw\equiv0$ at $t=0$, this inequality can now be integrated to produce a polynomial bound on $Iw$. Here we see the importance of $I$: it provides essential \textit{a priori} bounds while completely canceling out the self-similar background.

	The decomposition \eqref{eq:decomp_u_Szw} of $u$ is quite natural. If $S\equiv 0$, then $w$ is just the Duhamel integral term, for which one may indeed expect a polynomial bound by applying the $H^s$ local well-posedness theory. The point of this work is that self-similar solutions do not disrupt the classical theory, even though they do not belong to the usual spaces involved in the Cauchy problem. A solution with a self-similar background can still be obtained as a perturbation of the linear flow. 

\begin{nb}
	Speaking loosely, self-similar solutions appear from the underlying structure of the \textit{equation} and not from any specific balance between nonlinearity and dispersion (as it is for solitons). Their blow-up behavior is caused by the equation itself. Being unavoidable, it is should also be stable. This is in strong contrast with soliton-related blow-up, where the singularity comes from the precise structure of the \textit{solution}. There, small perturbations may obviously lead to strong unstable behavior.
\end{nb}

\bigskip

We now state the main result of this paper. Define the space of admissible perturbations
\begin{equation}
\mathcal{Z}:=\left\{z\in \mathcal{S}'(\R): \vvvert z \vvvert:=\|z\|_{L^1} + \|\jap{\xi}^{2}\hat{z}\|_{L^1} + \|\jap{\xi}\partial_\xi \hat{z}\|_{L^1} < +\infty \right\}.
\end{equation}
Here and below, $\jap{x} := \sqrt{1+|x|^2}$ stands for the japanese bracket.

\begin{thm}[Stability of self-similar solutions at blow-up time]\label{teo:existw}
	There exists $\delta>0$ such that, given $z\in \mathcal{Z}$ and a self-similar solution $S \in \mathcal{E}((0,+\infty))$ with 
	\begin{equation}\label{eq:small}
	\nz{z}+\|S\|_{\mathcal{E}((0,+\infty))}<\delta,
	\end{equation} 
	there exists a unique $w\in \mathcal{E}((0,+\infty)) \cap L^\infty(\R^+, L^2(\R))$ distributional solution to \eqref{eq:w} satisfying
\begin{equation}\label{eq:estfinal}
	\forall t >0, \quad \|w(t)\|_{\mathcal{E}(t)}\le \delta t^{1/9},\quad \|w(t)\|_{L^2(\R)}\le \delta^2 t^{1/18}.
\end{equation}
	In particular, $u(t)=S+e^{-t\partial_x^3}z+w$ is a distributional solution of \eqref{mKdV} on $\R^+\times \R$ satisfying
	\[
	\begin{cases}
u(t)-S(t) \to z \quad \text{in } L^2(\R)\\
\hat{u}(t)-\hat{S}(t) \to \hat{z} \quad \text{in } L^\infty(\R)
	\end{cases} \text{ as} \quad t \to 0^+. 
	\]
\end{thm}

\begin{nb}
	From time reversibility, one may solve the problem for negative times and glue the solutions together. Thus one may actually go \textit{beyond} the blow-up time. After some careful considerations, this is not that surprising: over $\mathcal{E}$, the self-similar solution does not present any sort of blow-up behavior at $t=0$. 
\end{nb}
In order to prove this result, we first need to understand how the various components of $u$ interact in the nonlinear term. This is done in Section \ref{sec:estimates}. Afterwards, in Section \ref{sec:aprox}, we construct an approximation sequence by cutting off high frequencies (Proposition \ref{lem:existwn}) and prove the necessary \textit{a priori} bounds in $\mathcal{E}$ through a careful bootstrap argument (Proposition \ref{prop:aprioriI}). Finally, in Section \ref{sec:proofthm}, the limiting procedure yields the claimed solution on a small time interval, which can then be extended for all positive times using the global results of \cite{CCV20}. The uniqueness statement follows from a direct energy argument (Proposition \ref{prop:uniq}).

%

\subsection{Acknowledgements}
We would like to thank Luis Vega for his encouragement and insightful remarks. S.C. was partially supported by Funda\c{c}\~ao para a Ci\^encia e Tecnologia, through CAMGSD, IST-ID
(projects UIDB/04459/2020 and UIDP/04459/2020) and through the project NoDES (PTDC/MAT-PUR/1788/2020).

\section{Linear and multilinear estimates}\label{sec:estimates}

In the following, the variables $\xi$, $\xi_1$, $\xi_2$ and $\xi_3$ are linked via the relation
\[ \xi = \xi_1 + \xi_2 + \xi_3. \]
We will perform a stationary phase analysis, with the phase
\begin{align} \label{def:Phi}
\Phi = \Phi(\xi,\xi_1,\xi_2) := \xi^3 - (\xi_1^3 + \xi_2^3+ \xi_3^3) = 3 (\xi-\xi_1) (\xi-\xi_2)(\xi-\xi_3).
\end{align}

Consider the trilinear version of $N$ defined by
\[ N[\tilde f, \tilde g, \tilde h](t,\xi) : = \frac{i \xi}{4\pi^2} \iint_{\xi_1+\xi_2+\xi_3=\xi} e^{it\Phi}\tilde{f}(t,\xi_1) \tilde{g}(t,\xi_2)\tilde{h}(t,\xi_3)d\xi_1 d\xi_2. \]

We state a trilinear version of Lemma \ref{lem:desenvolveoscilatorio}.

\begin{lem}[$L^\infty$ bounds in $\q E$] \label{lem:N_E}
For any $t >0$ and $f,g,h \in \q E(t)$,
\begin{equation} \label{est:N_E}
 |N[\tilde f, \tilde g, \tilde h](t,\xi) | \lesssim \frac{1}{t} \| f \|_{\q E(t)} \| g \|_{\q E(t)} \| h \|_{\q E(t)} .
\end{equation}
\end{lem}

\begin{proof}
This can be derived from polarizing \eqref{eq:expansaou}-\eqref{eq:N_remainder}. We refer to \cite[Lemma 7]{CCV20}: actually its proof (in the appendix there) is done for the trilinear version $N[\tilde f, \tilde g, \tilde h]$, and gives in particular \eqref{est:N_E}.
\end{proof}

The $1/t$ decay in \eqref{est:N_E} cannot be improved, in view of the leading terms in \eqref{eq:expansaou}. However, if one of the functions involved is better behaved, namely belongs to $\q Z$, we can gain some decay in time. This our next result.

\begin{lem}[$L^\infty$ bounds on terms with $z$]\label{lem:termos_com_z}
	For any $0<t \le 1$, $z\in \mathcal{Z}$ and $v\in \mathcal{E}(t)$, one has
	\begin{align}
	\label{eq:est_SSz}
	\left\| N[\tilde z, \tilde v, \tilde v](t) \right\|_{L^\infty_\xi} & \lesssim \frac{1}{t^{8/9}} \vvvert z \vvvert \|v\|_{\mathcal{E}(t)}^2, \\
\label{eq:est_Szz}
	\left\| N[\tilde z, \tilde z, \tilde v](t)  \right\|_{L^\infty_\xi} & \lesssim \frac{1}{t^{2/3}} \vvvert z \vvvert^2\|v\|_{\mathcal{E}(t)}, \\
	\label{eq:est_zzz}
	\left\| N[\tilde z, \tilde z, \tilde z](t)   \right\|_{L^\infty_\xi} & \lesssim  \vvvert z \vvvert^3.
	\end{align}
\end{lem}



\begin{proof}

Estimate \eqref{eq:est_zzz} is direct : we simply bound by
	\begin{align*}
	|N[\tilde z, \tilde z, \tilde z](t,\xi)| & \le \left(\int_{\xi_1+\xi_2+\xi_3 =\xi} (|\xi_1| + |\xi_2| + |\xi_3|) |\hat{z}(\xi_1)\hat{z}(\xi_2) \hat z(\xi_3) |d\xi_1d\xi_2\right) \\
	& \lesssim  \| \hat z \|_{L^1}^2 \| \xi \hat z \|_{L^\infty} \lesssim \nz{z}^3.
	\end{align*}
We now prove  \eqref{eq:est_SSz}, \eqref{eq:est_Szz} simultaneously. For each fixed $t \in (0,1]$ and $\xi \in \m R$, we split $\m R^2$ into several domains $\q A, \q B$, etc.. For each of them, we consider various cases depending on the relative size of the frequencies involved with respect to $t$ (of course, the implicit constants do not depend on $(t,\xi)$). 

To shorten notation, we denote 
\[ I_1 = N[\tilde z,\tilde v,\tilde v](t,\xi) \quad \text{and} \quad  I_2 = N[\tilde z,\tilde z,\tilde v](t,\xi), \]
and, if $\q D \subset \m R^2$, we denote $I_1(\q D)$, $I_2(\q D)$ the corresponding integral where the domain of integration is $\q D$ instead of $\m R^2$.

\bigskip

\emph{Case $\q A$}. Let $\q A : = \{ (\xi_1,\xi_2) \in \m R^2 :  |\xi_1| \ge \max(|\xi|,|\xi_2|)/100 \}$.

The bound in this case is direct. Indeed,
\begin{align*}
|I_1(\q A)| & \lesssim |\xi| \int_{|\xi_1| \ge |\xi|/10} |\hat z(\xi_1)| \left( \int_{|\xi_2| \le 10 |\xi_1|} d\xi_2 \right) d\xi_1  \| \tilde v \|_{L^\infty}^2 \lesssim \int_{\m R} |\xi_1|^2 |\hat z(\xi_1)| d\xi_1  \| \tilde v \|_{L^\infty}^2 \\
&  \lesssim \| \jap{\xi}^2 \hat z \|_{L^1} \| v \|_{\q E(t)}^2.
\end{align*}
Similarly,
\[ |I_2(\q A)| \lesssim |\xi| \int_{\q A} |\hat z(\xi_1) \hat z(\xi_2)| d\xi_1 d\xi_2 \| \tilde v \|_{L^\infty} \lesssim \| \jap{\xi}\hat z \|_{L^1} \| \hat z \|_{L^1} \| v \|_{\q E(t)} \lesssim \vvvert z \vvvert^2 \| v \|_{\q E(t)}. \]

\bigskip

\emph{Case $\q B$}. Let $\q B : = \{ (\xi_1,\xi_2) \in \m R^2 : |\xi| \ge \max(|\xi_2|/10, 10 |\xi_1|) \}$. Here we consider several subcases depending on the size of $t \xi^3$.

\textit{Step }($\q B.0$). If $|t\xi^3|<10^9$, then
\begin{align*}
|I_1(\q B)| \lesssim |\xi| \int_{|\xi_2| \le 10 |\xi|} |\tilde v| d \xi_2 \| \hat z \|_{L^1} \| \tilde v \|_{L^\infty} \lesssim |\xi|^2  \| \hat z \|_{L^1} \| \tilde v \|_{L^\infty}  \lesssim \frac{1}{t^{2/3}} \vvvert z \vvvert \| v \|_{\q E(t)}^2.
\end{align*}
We bound similarly
\[ |I_2(\q B)| \lesssim |\xi|\| \hat z \|_{L^1}^2 \| \tilde v \|_{L^\infty} \lesssim \frac{1}{t^{1/3}} \vvvert z \vvvert^2 \| v \|_{\q E(t)}. \]

For the remaining computations in Case $\q B$, we assume that
\[ |t \xi^3| \ge 10^9, \]
and we further split the domain $\q B$ by letting
\[ \q B_1 = \{ (\xi_1,\xi_2) \in \q B : ||\xi_3|-|\xi_2|| \le a \} \quad \text{and} \quad \q B_2 = \q B \setminus \q B_1, \]
for some $0 < a < |\xi|/10$ (depending on $\xi$) to be fixed later.
We will perform an integration by parts using
\[ e^{i t \Phi} = \partial_{\xi_j} (e^{i t \Phi}) \frac{1}{it  \partial_{\xi_j} \Phi}, \]
where $j=1,2$ and recall that
\[  \partial_{\xi_j} \Phi = 3(\xi_3^2 - \xi_j^2) = 3(\xi_3 + \xi_j)(\xi_3-\xi_j). \]
Notice that on $\q B$, $|\partial_{\xi_j\xi_j}^2 \Phi| \lesssim |\xi|$. Also, an extra care should be taken with the boundary terms, as $\tilde v$ may have a jump at frequency $0$. To this end, the domains of integration are meant to be deprived from the lines $\xi_2=0$ or $\xi_3=0$, while the boundary terms are always meant to contain the corresponding portion of these lines. This is why, throughout this proof, we change from the standard notation and denote by $\partial \Delta$ the boundary of $\Delta \setminus (\{ \xi_2=0 \} \cup \{ \xi_3=0 \})$. This does not weigh on the estimates, as we will use the $\| \tilde v \|_{L^\infty}$ bound to control the boundary terms.

\bigskip

\textit{Step }$(\q B.1)$. On $\q B_1$, we have $|\xi_2+\xi_3-\xi| = |\xi_1| \le |\xi|/10$ so that $|\xi_2+\xi_3| \ge 9|\xi|/10$. On the other side, $||\xi_2| -|\xi_3|| \le a \le |\xi|/10$ is small relative to $|\xi_2+\xi_3|$: this implies that $|\xi_2 - \xi_3| = ||\xi_2| -|\xi_3|| \le |\xi|/10$, and we infer
\[ | \xi_2 - \xi/2|, | \xi_3 - \xi/2| \le |\xi|/10. \]
As a consequence, $|\xi_3|- |\xi_1| \ge |\xi|/2 - |\xi|/5 \ge  |\xi|/4$ and so $|\partial_{\xi_1} \Phi| \gtrsim |\xi|^2$. Therefore, we perform an IBP with respect to $\xi_1$:
\begin{align*}
|I_1(\q B_1)| & \le  \left| \xi \int_{\q B_1}  e^{it\Phi} \partial_{\xi_1}\left( \frac{1}{it\partial_{\xi_1} \Phi} \hat z(\xi_1) \tilde{v}(t,\xi_3) \right)   \tilde{v}(t,\xi_2)  d\xi_1 d\xi_2 \right| \\
& \quad + |\xi| \int_{\partial \q B_1} \frac{1}{t |\partial_{\xi_1} \Phi|}  |\hat z(\xi_1)   \tilde{v}(t,\xi_2)  \tilde{v}(t,\xi_3)| d\sigma(\xi_1,\xi_2) \\
& \lesssim |\xi| \int_{\q B_1} \frac{|\xi|}{t |\xi|^4} |\hat z(\xi_1)| \| \tilde v \|_{L^\infty}^2  + \frac{1}{t|\xi|^2}  |\partial_\xi z(\xi_1)| \| \tilde  v \|_{L^\infty}^2 d\xi_1 d\xi_2  \\
& \quad + |\xi| \int_{\q B_1} \frac{1}{t|\xi|^2}  |z(\xi_1)| |\partial_\xi \tilde v(\xi_3)| \| \tilde v \|_{L^\infty} d\xi_1 d\xi_2  + |\xi| \int_{\partial \q B_1} \frac{1}{t |\xi|^2} |\hat z(\xi_1)| \| \tilde v \|_{L^\infty}^2 d\sigma(\xi_1,\xi_2).
\end{align*}
On $\q B_1$, for fixed $\xi_1$, $|\xi_2 - (\xi- \xi_1)/2| = |\xi_2 - \xi_3|/2 \le a/2$, so that 
\begin{align*}
|\xi| \int_{\q B_1} \frac{|\xi|}{t |\xi|^4} |\hat z(\xi_1)| d\xi_1 d\xi_2 & \lesssim  \frac{1}{t|\xi|^2} \int_{\m R}  |z(\xi_1)| \left( \int_{|\xi_2 - (\xi- \xi_1)/2| \le a/2} d\xi_2 \right) d\xi_1 \lesssim \frac{a}{t |\xi|^2} \| \hat z \|_{L^1}, \\
|\xi| \int_{\q B_1}  \frac{1}{t|\xi|^2}  |\partial_\xi z(\xi_1)| d\xi_1 d\xi_2 & \lesssim \frac{1}{t|\xi|} \int_{\m R}  |\partial_\xi z(\xi_1)| \left( \int_{|\xi_2 - (\xi- \xi_1)/2| \le a/2} d\xi_2 \right) d\xi_1 \lesssim \frac{a}{t |\xi|} \| \partial_\xi \hat z \|_{L^1}
\end{align*}
and
\begin{align*}
& |\xi| \int_{\q B_1}  \frac{1}{t|\xi|^2}  |z(\xi_1)| |\partial_\xi \tilde v(\xi_3)| d\xi_1 d\xi_2 \\\lesssim & \frac{1}{t |\xi|} \int_{\m R}  | \hat z(\xi_1)| \left( \int_{|\xi_2 - (\xi- \xi_1)/2| \le a/2} |\partial_\xi \tilde v(\xi- \xi_1 - \xi_2)|  d\xi_2 \right) d\xi_1 \\
 \lesssim & \frac{1}{t|\xi|} \| \hat z \|_{L^1} a^{1/2} \| \partial_\xi \tilde v \|_{L^2} \lesssim \frac{a^{1/2}}{t^{5/6} |\xi|} \| \hat z \|_{L^1} \| v \|_{\q E(t)}.
\end{align*}
We see on the second bound that one requires $a \ll |\xi|$ in order to gain over the $1/t$ bound.

For the boundary term, we have
\begin{align*}
 |\xi| \int_{\partial \q B_1} \frac{1}{t |\xi|^2} |\hat z(\xi_1)| d\sigma(\xi_1,\xi_2) \lesssim \frac{1}{t|\xi|} \| \hat z \|_{L^1}.
\end{align*}
Therefore,
\begin{align} \label{est:I1_B1} 
|I_1(\q B_1)| \lesssim  \left(  \frac{a}{t |\xi|} +  \frac{a^{1/2}}{t^{5/6} |\xi|} +\frac{1}{t|\xi|} \right) \| \hat z \|_{W^{1,1}} \| v \|_{\q E(t)}^2.
\end{align}

\bigskip

\textit{Step }$(\q B.2)$. On $\q B_2$, $||\xi_2| - |\xi_3|| \ge a$. Also, as $|\xi_1| \le |\xi|/10$, $|\xi_2| + |\xi_3| \ge 9|\xi|/10$ and so $|\partial_{\xi_2} \Phi| \gtrsim a |\xi|$. Here, we perform an IBP in $\xi_2$:
\begin{align*}
|I_1(\q B_2)| & \le  \left| \xi \int_{\q B_2}  e^{it\Phi} \partial_{\xi_2}\left( \frac{1}{it\partial_{\xi_2} \Phi}  \tilde{v}(t,\xi_2)  \tilde{v}(t,\xi_3) \right)  \hat z(\xi_1)  d\xi_1 d\xi_2 \right| \\
& \quad + |\xi| \int_{\partial \q B_2} \frac{1}{t |\partial_{\xi_2} \Phi|}  |\hat z(\xi_1)   \tilde{v}(t,\xi_2)  \tilde{v}(t,\xi_3)| d\sigma(\xi_1,\xi_2) \\
& \lesssim |\xi| \int_{\q B_2}  \frac{|\xi|}{t |a \xi|^2}  |\hat z(\xi_1)| \| \tilde v \|_{L^\infty}^2 + \frac{1}{t|a\xi|}  |\hat z(\xi_1)| |\partial_\xi v(\xi_2)|  \| \tilde v \|_{L^\infty} d\xi_1 d\xi_2 \\
& \quad +|\xi|\int_{\q B_2}  \frac{1}{t|a \xi|}  |\hat z(\xi_1)|  \| \tilde  v \|_{L^\infty} |\partial_\xi \tilde v(\xi_3)|d\xi_1d\xi_2 + |\xi| \int_{\partial \q B_2} \frac{1}{t |a\xi|} |\hat z(\xi_1)| \| \tilde v \|_{L^\infty}^2 d\sigma(\xi_1,\xi_2)
\end{align*}
Observe that all derivatives fall on $\tilde v$ (or $\Phi$, but not $\hat z$), the point being that $\| \partial_\xi \tilde v \|_{L^2}$ is better behaved than $\| \tilde v \|_{L^\infty}$. To complete the bounds, we now only use that $|\xi_1|, |\xi_2| \lesssim |\xi|$ on $\q B_2$ as follows
\begin{align*}
\MoveEqLeft |\xi| \int_{\q B_2} \frac{1}{t a^2 |\xi|} |\hat z(\xi_1)| d\xi_1 d\xi_2 \lesssim \frac{|\xi|}{t a^2} \| \hat z \|_{L^1}, \\
\MoveEqLeft |\xi| \int_{\q B_1}  \frac{1}{t|a\xi|}  |z(\xi_1)| (|\partial_\xi \tilde v(\xi_2)| + |\partial_\xi \tilde v(\xi_3)| d\xi_1 d\xi_2 \\
& \lesssim \frac{1}{t a} \int_{\m R}  | \hat z(\xi_1)| \left( \int_{|\xi_2| \lesssim |\xi|} (|\partial_\xi v(\xi_2)| +| \partial_\xi \tilde v(\xi- \xi_1 - \xi_2)|  d\xi_2 \right) d\xi_1 \\
& \lesssim  \frac{1}{ta} \| \hat z \|_{L^1} |\xi|^{1/2} \| \partial_\xi \tilde v \|_{L^2} \lesssim \frac{|\xi|^{1/2}}{t^{5/6} |a|} \| \hat z \|_{L^1} \| v \|_{\q E(t)}.
\end{align*}
For the boundary term, we simply have
\[ \int_{\partial \q B_2} \frac{1}{t a} |\hat z(\xi_1)|  d\sigma(\xi_1,\xi_2) \lesssim \frac{1}{ta} \| \hat z \|_{L^1}. \]
Therefore,
\begin{align} \label{est:I1_B2}
I_1(\q B_2)| \lesssim \left(  \frac{|\xi|}{t a^2} +  \frac{|\xi|^{1/2}}{t^{5/6} a} + \frac{1}{ta} \right) \| \hat z \|_{W^{1,1}} \| v \|_{\q E(t)}^2.
 \end{align}

\bigskip

\textit{Step }$(\q B.3)$. We now optimize in $a$, choosing $a = |\xi|^{2/3}$. As $|\xi|^{1/3} \ge 10 t^{-1/9} \ge 10$, $a \le |\xi|/10$, which justifies the above computations. Using \eqref{est:I1_B1} and \eqref{est:I1_B2}, and  that $|\xi|^{-1} \lesssim t^{1/3}$, we get
\begin{align} \label{est:I1_B}
 |I_1(\q B)| \le |I_1(\q B_1)| + |I_1(\q B_2)| \lesssim \frac{1}{t^{8/9}} \| \hat z \|_{W^{1,1}} \| v \|_{\q E(t)}^2.
 \end{align}

\bigskip

 We now bound $I_2(\q B)$. The bounds are obtained in a similar fashion as for $I_1(\q B)$ (they are in fact simpler). However, to sharpen the bound, the frequency splitting is slightly different:
\[ \q B_4 = \{ (\xi_1,\xi_2) \in \q B : ||\xi_3|-|\xi_2|| \le |\xi|/10 \} \quad \text{and} \quad \q B_5 = \q B \setminus \q B_4, \]
(this corresponds to the choice $a = |\xi|/10$).

\bigskip

\textit{Step }$(\q B.4)$. For $I_2(\q B_4)$, as $|\partial_{\xi_1} \Phi| \gtrsim |\xi|^2$ on $\q B_4$, we perform an  IBP in $\xi_1$:
\begin{align*}
|I_2(\q B_4)| & \le  \left| \xi \int_{\q B_4}  e^{it\Phi} \partial_{\xi_1}\left( \frac{1}{it\partial_{\xi_1} \Phi} \hat z(\xi_1) \tilde{v}(t,\xi_3) \right) \hat z(\xi_2)  d\xi_1 d\xi_2 \right| \\
& \quad + |\xi| \int_{\partial \q B_4} \frac{1}{t |\partial_{\xi_1} \Phi|}  |\hat z(\xi_1) \hat z(\xi_2)  \tilde{v}(t,\xi_3)| d\sigma(\xi_1,\xi_2) \\
& \lesssim |\xi| \int_{\q B_4} \frac{|\xi|}{t |\xi|^4} |\hat z(\xi_1)| |\hat z(\xi_2)| \| \tilde v \|_{L^\infty}  + \frac{1}{t|\xi|^2}  |\partial_\xi \hat z(\xi_1) \hat z(\xi_2)| \| \tilde  v \|_{L^\infty}  d\xi_1 d\xi_2 \\
& \quad + |\xi|\int_{\q B_4}  \frac{1}{t|\xi|^2}  |\hat z(\xi_1) \hat z (\xi_2) \partial_\xi \tilde v(\xi_3) | d\xi_1d\xi_2\\
& \quad + |\xi| \int_{\partial \q B_4} \frac{1}{t |\xi|^2} |\hat z(\xi_1)|  \| \hat z \|_{L^\infty} \| \tilde v \|_{L^\infty} d\sigma(\xi_1,\xi_2).
\end{align*}
The gain over the case $(\q B.1)$ comes from the two factors in $z$ which insure an $L^1(d\xi_1d\xi_2)$ bound:
\begin{align*}
|I_2(\q B_4)| & \lesssim \frac{1}{t|\xi|^2} \| \hat z \|_{L^1}^2 \| \tilde v \|_{L^\infty} + \frac{1}{t|\xi|} \| \partial_\xi \hat z \|_{L^1} \| \hat z \|_{L^1} \| \tilde v \|_{L^\infty} \\ & \qquad + \frac{1}{t|\xi|} \| \hat z \|_{L^1} \| \partial_\xi v \|_{L^2} \| \hat z \|_{L^2} + \frac{1}{t |\xi|} \|\hat z \|_{L^1} \| \hat z \|_{L^\infty} \| \tilde v \|_{L^\infty} \\
& \lesssim \frac{1}{t|\xi|} \|\hat z\|_{L^1} \| \hat z \|_{W^{1,1}} \| v \|_{\q E(t)}.
\end{align*}
As we assumed $|\xi| \gtrsim t^{-1/3}$ here, we infer
\begin{align} \label{est:I2_B1}
|I_2(\q B_4)| & \lesssim \frac{1}{t^{2/3}} \| \hat z \|_{L^1} \| \hat z \|_{W^{1,1}} \| v \|_{\q E(t)}.
\end{align}

\textit{Step }$(\q B.5)$. For $I_2(\q B_5)$, $|\partial_{\xi_2} \Phi| \gtrsim |\xi|^2$, so that we perform an IBP in $\xi_2$:
\begin{align*}
\MoveEqLeft |I_2(\q B_5)|  \le  \left| \xi \int_{\q B_5}  e^{it\Phi} \partial_{\xi_2}\left( \frac{1}{it\partial_{\xi_2} \Phi}  \hat z(\xi_2)  \tilde{v}(t,\xi_3) \right)  \hat z(\xi_1)  d\xi_1 d\xi_2 \right| \\
& \qquad + |\xi| \int_{\partial \q B_5} \frac{1}{t |\partial_{\xi_1} \Phi|}  |\hat z(\xi_1)  |\hat {z}(\xi_2)  \tilde{v}(t,\xi_3)| d\sigma(\xi_1,\xi_2) \\
& \lesssim |\xi| \int_{\q B_5}  \frac{|\xi|}{t | \xi|^4}  |\hat z(\xi_1)|  |\hat z(\xi_2)| \| \tilde v \|_{L^\infty} + \frac{1}{t|\xi|^2}  |\hat z(\xi_1)| |\partial_\xi \hat z(\xi_2)|  \| \tilde v \|_{L^\infty} d\xi_1d\xi_2 \\
& \qquad + |\xi| \int_{\q B_5}\frac{1}{t|\xi|^2}  |\hat z(\xi_1)|  |\hat z(\xi_2)|  |\partial_\xi \tilde v(\xi_3)|  d\xi_1 d\xi_2 \\&\qquad+ |\xi| \int_{\partial \q B_5} \frac{1}{t |\xi|^2} |\hat z(\xi_1)| \|\hat z \|_{L^\infty} \| \tilde v \|_{L^\infty} d\sigma(\xi_1,\xi_2) \\
& \lesssim \frac{1}{t |\xi|^2} \| \hat z \|_{L^1}^2  \| \tilde v \|_{L^\infty} + \frac{1}{t|\xi|} \| \hat z \|_{L^1} \| \partial_\xi \hat z \|_{L^1} \| v \|_{L^\infty} \\& \qquad + \frac{1}{t|\xi|} \| \hat z \|_{L^1} \| \hat z \|_{L^2} \| \partial_\xi \tilde v \|_{L^2}   + \frac{1}{t|\xi|}\|\hat z\|_{L^1}\|\hat z\|_{L^\infty}\|\tilde{v}\|_{L^{\infty}}\\
& \lesssim \frac{1}{t|\xi|} \| \hat z \|_{L^1} \| \hat z \|_{W^{1,1}} \| \tilde v \|_{\q E(t)}.
\end{align*}
(recall that $0 <t \le 1$). Together with \eqref{est:I2_B1}, we infer
\begin{align} \label{est:I2_B}
|I_2(\q B)| \le |I_2(\q B_4)| + |I_2(\q B_5)| \lesssim \frac{1}{t^{2/3}}  \| \hat z \|_{L^1} \| \hat z \|_{W^{1,1}} \| \tilde v \|_{\q E(t)}.
\end{align}

\bigskip

\emph{Case $\q C$.} We finally consider
\[ \q C = \m R^2 \setminus (\q A \cup \q B) = \{ (\xi_1,\xi_2) \in \m R^2 : |\xi_1| < \max(|\xi|,|\xi_2|)/100, \ |\xi| < \max( |\xi_2|/10, 10 |\xi_1|) \}. \]
Observe that on $\q C$, $|\xi| \le \max( |\xi_2|/10, \max(|\xi|,  |\xi_2|)/10) = \max(|\xi|/10, |\xi_2|/10)$ so that $|\xi| \le |\xi_2|/10$, and therefore $|\xi_1| \le |\xi_2|/100$.
Hence $|\xi_3+\xi_2| \le |\xi_2| /5$, $|\xi_3| \ge 4|\xi_2|/5$ and $\xi_2$ and $\xi_3$ are the highest frequencies (of the same magnitude). In particular, $|\partial_{\xi_i \xi_j}^2 \Phi| \lesssim |\xi_2|$ on $\q C$.

We argue in $\q C$ in the same spirit as we did for case $\q B$. We split
\begin{align*}
\q C_0 &= \{ (\xi_1, \xi_2) \in \q C : |t\xi_2^3| \le 10^9 \} \\
\q C_1 &= \{ (\xi_1, \xi_2) \in \q C : |t\xi_2^3| \ge 10^9 \text{ and } |\xi-\xi_1| \le |\xi_2|^{2/3} \}, \\
\q C_2 &= \{ (\xi_1, \xi_2) \in \q C : |t\xi_2^3| \ge 10^9 \text{ and }  |\xi-\xi_1| \ge |\xi_2|^{2/3} \}.
\end{align*}
($\xi_2$ is now playing the role of $\xi$ in Case $\q B$).

\bigskip 

\textit{Step }$(\q C.0)$. On $\q C_0$, we bound as in $(\q B.0)$:
\begin{align}
I_1(\q C_0) & \le |\xi| \int_{\q C_0} |\hat z(\xi_1)|  d\xi_1 d\xi_2 \| \tilde v \|_{L^\infty}^2 \lesssim \| \hat z \|_{L^1}  \| \tilde v \|_{L^\infty}^2 \int_{|\xi_2| \lesssim t^{-1/3}} |\xi_2| d\xi_2 \\
&  \lesssim \frac{1}{t^{2/3}}  \| \hat z \|_{L^1}  \| \tilde v \|_{L^\infty}^2, \label{est:I1_C0} \\
I_2(\q C_0) & \le |\xi| \int_{\q C_0} |\hat z(\xi_1)|   d\xi_1 d\xi_2 \| \hat z \|_{L^\infty} \| \tilde v \|_{L^\infty} \lesssim \| \hat z \|_{L^1}   \| \hat z \|_{L^\infty}  \| \tilde v \|_{L^\infty} \int_{|\xi_2| \lesssim t^{-1/3}} |\xi_2| d\xi_2 \\
& \lesssim \frac{1}{t^{2/3}}  \| \hat z \|_{L^1}  \| \hat z \|_{L^\infty} \| \tilde v \|_{L^\infty}. \label{est:I2_C0} 
\end{align}

\bigskip

\textit{Step }$(\q C.1)$. On $\q C_1$, we integrate by parts in $\xi_1$: observe that in this domain $|\xi_1| \le |\xi_2|/100$ and $|\xi_3| \ge 4|\xi_2|/5$ so that
\[ |\partial_{\xi_1} \Phi| = 3 |\xi_1^2-\xi_3^2| \gtrsim |\xi_2|^2. \]
Hence,
\begin{align*}
|I_1(\q C_1)| & \le  \left| \xi \int_{\q C_1}  e^{it\Phi} \partial_{\xi_1}\left( \frac{1}{it\partial_{\xi_1} \Phi} \hat z(\xi_1) \tilde{v}(t,\xi_3) \right)   \tilde{v}(t,\xi_2)  d\xi_1 d\xi_2 \right| \\
& \qquad + |\xi| \int_{\partial \q C_1} \frac{1}{t |\partial_{\xi_1} \Phi|}  |\hat z(\xi_1)   \tilde{v}(t,\xi_2)  \tilde{v}(t,\xi_3)| d\sigma(\xi_1,\xi_2) \\
& \lesssim |\xi| \int_{\q C_1} \frac{|\xi_2|}{t |\xi_2|^4} |\hat z(\xi_1)| \| \tilde v \|_{L^\infty}^2  + \frac{1}{t|\xi_2|^2}  |\partial_\xi \hat z(\xi_1)| \| \tilde  v \|_{L^\infty}^2 d\xi_1d\xi_2 \\&\qquad+|\xi|\int_{\q C_1}  \frac{1}{t|\xi_2|^2}  |z(\xi_1)| |\partial_\xi \tilde v(\xi_3)| \| \tilde v \|_{L^\infty} d\xi_1 d\xi_2 \\
& \qquad + |\xi| \int_{\partial \q C_1} \frac{1}{t |\xi_2|^2} |\hat z(\xi_1)| \| \tilde v \|_{L^\infty}^2 d\sigma(\xi_1,\xi_2).
\end{align*}
On $\q C_1$, $|\xi| \lesssim |\xi_2|$, and for fixed $\xi_2$, $|\xi-\xi_1|  \le |\xi_2|^{2/3}$, so that 
\begin{align*}
|\xi| \int_{\q C_1} \frac{|\xi_2|}{t |\xi_2|^4} |\hat z(\xi_1)| d\xi_1 d\xi_2 & \lesssim \frac{1}{t} \int_{|\xi_2| \gtrsim t^{-1/3}} \frac{d\xi_2}{|\xi_2|^2} \| \hat z \|_{L^1} \lesssim \frac{1}{t^{2/3}}   \| \hat z \|_{L^1}, \\
|\xi| \int_{\q C_1}  \frac{1}{t|\xi_2|^2}  |\partial_\xi z(\xi_1)| d\xi_1 d\xi_2 & \lesssim \frac{1}{t} \int_{|\xi_2|  \gtrsim t^{-1/3}} \int  \left( |\xi-\xi_1| + |\xi_1| |\partial_\xi z(\xi_1)| d\xi_1 \right) \frac{d\xi_2}{|\xi_2|^2}  \\
& \lesssim  \frac{1}{t}  \int_{|\xi_2| \gtrsim t^{-1/3}} \left( |\xi_2|^{2/3} \| \partial_\xi \hat z \|_{L^1} + \| \xi \partial_\xi \hat z \|_{L^1} \right) \frac{d\xi_2}{|\xi_2|^2} \\&\lesssim \frac{1}{t^{8/9}} \| \jap{\xi} \partial_\xi \hat z \|_{L^1}, \\
|\xi| \int_{\q C_1}  \frac{1}{t|\xi_2|^2}  |z(\xi_1)| |\partial_\xi \tilde v(\xi_3)| d\xi_1 d\xi_2 & \lesssim \frac{1}{t} \int_{\m R}  | \hat z(\xi_1)| \left( \int_{|\xi_2| \gtrsim t^{-1/3}} |\partial_\xi \tilde v(\xi_3)|  \frac{d\xi_2}{|\xi_2|} \right) d\xi_1 \\
& \lesssim \frac{1}{t} \| \hat z \|_{L^1} \| \partial_\xi \tilde v \|_{L^2} \left( \int_{|\xi_2| \gtrsim t^{-1/3}} \frac{d\xi_2}{|\xi_2|^2} \right)^{1/2} \\&\lesssim \frac{1}{t^{5/6}} \| \hat z \|_{L^1} \| \partial_\xi \tilde v \|_{L^2} \lesssim \frac{1}{t^{2/3}}\|\hat z\|_{L^1}\|v\|_{\q E (t)}.
\end{align*} 
On $\partial \q C_1$, since $|\xi_2| \gtrsim |\xi|$ and $|\xi_2| \gtrsim t^{-1/3}$, one has $\ds \frac{|\xi|}{ |\xi_2|^2} \lesssim t^{1/3}$. Thus
\begin{equation} \label{est:dC_1}
|\xi| \int_{\partial \q C_1} \frac{1}{t |\xi_2|^2} |\hat z(\xi_1)| d\sigma(\xi_1,\xi_2) \le \frac{1}{t^{2/3}} \| \hat z \|_{L^1},
\end{equation}
Therefore, we get
\begin{align} \label{est:I1_C1}
|I_1(\q C_1)| \lesssim \frac{1}{t^{8/9}} (\| \jap{\xi} \partial_\xi \hat z \|_{L^1} + \| \hat z \|_{L^1})  \| v \|_{\q E(t)}^2.
\end{align}

\bigskip

\textit{Step }$(\q C.2)$. On $\q C_2$, we integrate by parts in $\xi_2$: observe that in this domain
\[ |\partial_{\xi_2} \Phi| = 3 |\xi_2^2-\xi_3^2| = 3 |\xi_2- \xi_3| |\xi-\xi_1| \gtrsim |\xi_2|^{5/3},\]
Hence we can estimate
\begin{align*}
|I_1(\q C_2)| & \le \left| \xi \int_{\q C_2}  e^{it\Phi} \partial_{\xi_1}\left( \frac{1}{it\partial_{\xi_2} \Phi} \tilde{v}(t,\xi_2)  \tilde{v}(t,\xi_3) \right) \hat z(\xi_1) d\xi_1 d\xi_2 \right| \\
& \qquad + |\xi| \int_{\partial \q C_2} \frac{1}{t |\partial_{\xi_2} \Phi|}  |\hat z(\xi_1)   \tilde{v}(t,\xi_2)  \tilde{v}(t,\xi_3)| d\sigma(\xi_1,\xi_2) \\
& \lesssim |\xi| \int_{\q C_2}  \frac{|\xi_2|}{t |\xi_2|^{10/3}} |\hat z(\xi_1)| \| \tilde v \|_{L^\infty}^2 d\xi_1d\xi_2\\&\qquad  + |\xi|\int_{\q C_2}\frac{1}{t|\xi_2|^{5/3}}  |\hat z(\xi_1)| (|\partial_\xi \tilde  v(\xi_2)| +|\partial_\xi \tilde  v(\xi_3)|) \| \tilde v \|_{L^\infty} d\xi_1 d\xi_2 \\
& \qquad + |\xi| \int_{\partial \q C_2} \frac{1}{t |\xi_2|^{5/3}} |\hat z(\xi_1)| \| \tilde v \|_{L^\infty}^2 d\sigma(\xi_1,\xi_2).
\end{align*}
On $\q C_2$, $|\xi| \lesssim |\xi_2|$, so that 
\begin{align*}
\MoveEqLeft |\xi| \int_{\q C_2} \frac{1}{t |\xi_2|^{7/3}} |\hat z(\xi_1)| d\xi_1 d\xi_2 \lesssim \frac{1}{t} \int_{|\xi_2| \gtrsim t^{-1/3}} \frac{d\xi_2}{|\xi_2|^{4/3}} \| \hat z \|_{L^1} \lesssim \frac{1}{t^{8/9}}   \| \hat z \|_{L^1}, \\
\MoveEqLeft  |\xi| \int_{\q C_2}  \frac{1}{t|\xi_2|^{5/3}}  |\hat z(\xi_1) \partial_\xi \tilde  v(\xi_2)| d\xi_1 d\xi_2 \lesssim \frac{1}{t} \int_{\m R} |\hat z(\xi_1)| \left( \int_{|\xi_2|  \gtrsim t^{-1/3}} | \partial_\xi \tilde  v(\xi_2)|  \frac{d\xi_2}{|\xi_2|^{2/3}} \right) \\
& \lesssim \frac{1}{t} \| \hat z \|_{L^1} \|  \partial_\xi \tilde  v \|_{L^2} \left( \int_{|\xi_2|  \gtrsim t^{-1/3}} \frac{d\xi_2}{|\xi_2|^{4/3}} \right)^{1/2} \lesssim \frac{1}{t} \| \hat z \|_{L^1} \|  \partial_\xi \tilde  v \|_{L^2} t^{1/18} \lesssim \frac{1}{t^{7/9}}  \| \hat z \|_{L^1} \| v \|_{\q E(t)}.
\end{align*}
On $\partial \q C_2$, we have, as in \eqref{est:dC_1}, $\ds \frac{|\xi|}{|\xi_2|^{5/3}} \lesssim t^{2/9}$ so that
\begin{equation} \label{est:dC_2}
|\xi| \int_{\partial \q C_2} \frac{1}{t |\xi_2|^{5/3}} |\hat z(\xi_1)| d\sigma(\xi_1,\xi_2) \le \frac{1}{t^{7/9}} \| \hat z \|_{L^1}.
\end{equation}
Hence, we obtain
\[ |I_1(\q C_2)|  \lesssim \frac{1}{t^{8/9}} \| \hat z \|_{L^1} \| v \|_{\q E(t)}^2. \]
Together with \eqref{est:I1_C0} and \eqref{est:I1_C1}, we infer that
\[ |I_1(\q C)| \lesssim \frac{1}{t^{8/9}} \vvvert z \vvvert \| v \|_{\q E(t)}^2. \]

\bigskip

\textit{Step }$(\q C.3)$. We now bound $I_2$ on 
\[ \q C_3 = \q C \setminus \q C_0 = \{ (\xi_1,\xi_2) \in \q C : |t \xi_2|^3 > 10^9 \}. \]
For $I_2$, we don't need to further split the domain. As for $\q C_1$, on $\q C_3$ we integrate by parts with respect to $\xi_1$. In this region, $|\xi_3| \ge 4|\xi_2|/5$ and $|\xi_1| \le |\xi_2|/5$ so that $|\partial_{\xi_1} \Phi| = 3 |\xi_1^2 - \xi_3|^2 \gtrsim |\xi_2|^2$. Hence we bound:
\begin{align*}
|I_2(\q C_3)| & \le \left| \xi \int_{\q C_3}  e^{it\Phi} \partial_{\xi_1}\left( \frac{1}{it\partial_{\xi_1} \Phi} \hat z(\xi_1)  \tilde{v}(t,\xi_3) \right) \hat z(\xi_2) d\xi_1 d\xi_2 \right| \\
& \qquad + |\xi| \int_{\partial \q C_3} \frac{1}{t |\partial_{\xi_1} \Phi|}  |\hat z(\xi_1)  \hat z(\xi_2)  \tilde{v}(t,\xi_3)| d\sigma(\xi_1,\xi_2) \\
& \lesssim |\xi| \int_{\q C_3}  \frac{|\xi_2|}{t |\xi_2|^{4}} |\hat z(\xi_1) \hat z(\xi_2)| \| \tilde  v \|_{L^\infty}  + \frac{1}{t|\xi_2|^{2}}  |\partial_\xi \hat z(\xi_1)  \hat  z(\xi_2)| \| \tilde v \|_{L^\infty}d\xi_1d\xi_2\\&\qquad +  \int_{\q C_3} \frac{1}{t|\xi_2|^{2}}  |\hat z(\xi_1) \hat z(\xi_2) \partial_\xi \tilde  v(\xi_3)|  d\xi_1 d\xi_2 \\
& \qquad + |\xi| \int_{\partial \q C_3} \frac{1}{t |\xi_2|^{2}} |\hat z(\xi_1)| \| \hat z \|_{L^\infty} \| \tilde v \|_{L^\infty} d\sigma(\xi_1,\xi_2).
\end{align*}
On $\q C_3$, $|\xi| \lesssim |\xi_2|$ and $\frac{1}{|\xi_2|}\lesssim t^{1/3}$. Hence
\begin{align*}
 |\xi| \int_{\q C_3} \frac{|\xi_2|}{t |\xi_2|^{4}} |\hat z(\xi_1) \hat z(\xi_2)| d\xi_1 d\xi_2 &\lesssim \frac{1}{t^{1/3}} \int |\hat z(\xi_1) \hat z(\xi_2)| d\xi_1 d\xi_2\lesssim  \frac{1}{t^{1/3}} \| \hat z \|_{L^1}^2, \\
 |\xi| \int_{\q C_3}  \frac{1}{t|\xi_2|^{2}}  | \partial_\xi \hat z(\xi_1) z(\xi_2)|  d\xi_1 d\xi_2 &\lesssim \frac{1}{t^{2/3}} \int_{\m R} |\partial_\xi \hat z(\xi_1) z(\xi_2)| d\xi_1 d\xi_2 \lesssim  \frac{1}{t^{2/3}} \|\partial_\xi \hat z \|_{L^1} \| \hat z \|_{L^1}
 \end{align*}
 and
 \begin{align*}
  |\xi| \int_{\q C_3}  \frac{1}{t|\xi_2|^{2}}  |\hat z(\xi_1) z(\xi_2) \partial_\xi \tilde v(t,\xi_3)|  d\xi_1 d\xi_2 &\lesssim \frac{1}{t^{2/3}} \int |\hat z(\xi_1)| \left(  \int |z(\xi_2) \partial_\xi \tilde v(t,\xi_3)| \right) d\xi_2 \\
 \lesssim \frac{1}{t^{2/3}} \| \hat z \|_{L^1} \| \hat z \|_{L^2} \| \partial_\xi \tilde v \|_{L^2} &\lesssim \frac{1}{t^{1/2}} \| \hat z \|_{L^1} \| \hat z \|_{L^2} \| v \|_{\q E(t)}.
\end{align*}
On $\partial \q C_3$, we have similarly 
\begin{equation} \label{est:dC_3}
|\xi| \int_{\partial \q C_3} \frac{1}{t |\xi_2|^{2}} |\hat z(\xi_1)| d\sigma(\xi_1,\xi_2) \le \frac{1}{t^{2/3}} \| \hat z \|_{L^1}.
\end{equation}
Therefore
\[ |I_2(\q C_3)| \lesssim \frac{1}{t^{2/3}} \|  \hat z \|_{L^1} \| \hat z \|_{W^{1,1}} \| v \|_{\q E(t)}. \]

Together with \eqref{est:I2_C0}, we conclude that
\[ |I_2(\q C)| \lesssim \frac{1}{t^{2/3}} \|  \hat z \|_{L^1} \| \hat z \|_{W^{1,1}} \| v \|_{\q E(t)}. \]

\bigskip

\emph{Conclusion}.

Summing up the bounds obtained in case $\q A$, $\q B$ and $\q C$, and observing that these cover $\m R^2$, we conclude that
\begin{align*}
|I_1| & \lesssim \frac{1}{t^{8/9}}  \| \jap{\xi} \hat z \|_{W^{1,1}} \| v \|_{\q E(t)}^2. \\
|I_2| & \lesssim \frac{1}{t^{2/3}}  \| \hat z \|_{L^1} \| \hat z \|_{W^{1,1}} \| v \|_{\q E(t)}.  \qedhere
\end{align*}
\end{proof}

It turns out in the energy estimates involving the dilation operator $I$, some terms can not be interpreted as a full derivative (mainly because the equation for $w$ has a source term). In addition to \eqref{est:N_E}, we will also need the following bound on a term of the form $N[\tilde S, \tilde S,v]$, but where the weight $\xi$ (associated to the derivative in physical space) only falls on the $v$ term.

\begin{lem}\label{lem:estL2}
	For any $0<t<1$, if $v\in \mathcal{E}(t)$,
	\[
	\left\| \iint e^{it\Phi}\tilde{S}(t^{1/3}\xi_1)\tilde{S}(t^{1/3}\xi_2)\xi_3\tilde{v}(\xi_3)d\xi_1 d\xi_2\right\|_{L^\infty_\xi}\lesssim \frac{1}{t}\|S\|_{\mathcal{E}(t)}^2\|v\|_{\mathcal{E}(t)}.
	\]
\end{lem}
\begin{nb}
	This estimate involves only the critical norm $\q E(t)$. Consequently, the $1/t$ decay is optimal.
\end{nb}

\begin{proof}
Before we perform the estimate, we rewrite the oscillatory integral. Set $\eta=\xi_1+\xi_2$, $\nu=\xi_1-\xi_2$ and
\[
\Psi=-3\left(\eta\xi^2 - \eta^2\xi+\frac{1}{4}\eta^3\right).
\]
Then a simple computation yields
\begin{align*}
\MoveEqLeft I(t,\xi) : = \int e^{it\Phi}\tilde{S}(t^{1/3}\xi_1)\tilde{S}(t^{1/3}\xi_2)\xi_3\tilde{v}(\xi_3)d\xi_1 d\xi_2 \\
& = \int e^{it\Psi} (\xi- \eta) \tilde{v}(\xi - \eta)\left(\int e^{3it\eta\nu^2/4}\tilde{S}\left(t^{1/3}\frac{\eta+\nu}{2}\right)\tilde{S}\left(t^{1/3}\frac{\eta-\nu}{2}\right)d\nu\right)d\eta \\
& = \frac{1}{t^{1/3}}\int e^{it\Psi}\xi_3 \tilde{v}(\xi_3)\left(\int e^{3i t^{1/3} \eta\mu^2/4}\tilde{S}\left(\frac{t^{1/3}\eta+\mu}{2}\right)\tilde{S}\left(\frac{t^{1/3}\eta-\mu}{2}\right)d\mu\right)d\eta \\
& = : \frac{1}{t^{1/3}}\int e^{it\Psi}(\xi-\eta) \tilde{v}(\xi-\eta) K(S,S)(t^{1/3}\eta) d\eta
\end{align*}
where the function 
\[ K(S,S)(\sigma) : = \int e^{3i \sigma \mu^2/4}\tilde{S}\left(\frac{\sigma+\mu}{2}\right)\tilde{S}\left(\frac{\sigma-\mu}{2}\right)d\mu \]
has been studied in \cite[Lemma 14]{CCV19} and satisfies the bounds
\[
|K(S,S)(t^{1/3}\eta)|\lesssim \frac{\|S\|_{\mathcal{E}(t)}^2}{t^{1/6}|\eta|^{1/2}},\quad \left|\partial_\eta \left(K(S,S)(t^{1/3}\eta)\right)\right|\lesssim \frac{\|S\|_{\mathcal{E}(t)}^2}{t^{1/6}|\eta|^{3/2}}.
\]
As
\[
\frac{\partial \Psi}{\partial \eta} = -\frac{3}{4}(\eta-2\xi)(3\eta-2\xi),
\]
the phase $\Psi$ has two stationary points, $\eta=2\xi$ and $\eta = 2\xi/3$. We use the same notation as in the previous proof: given a set $\Delta\subset\R$, $I(\Delta)$ is the restriction of $I$ to $\Delta$. We perform a stationary phase analysis in the complementary regions
\begin{gather*}
\Delta_1 =  \{ \eta \in \m R : |\eta| \le 10 |\xi| \}  \quad \text{and} \quad  \Delta_2 = \{ \eta \in \m R : |\eta| \ge 10 |\xi| \}.
\end{gather*}

\

\emph{Step 1.} We first consider the case $|t^{1/3} \xi| \le 1$.

(i) On $\Delta_1$, $|\xi - \eta| \lesssim |\xi|$ so that
\[ |I(\Delta_1)| \lesssim  \frac{1}{t^{1/2}}\int_{|\eta|\lesssim |\xi|} \frac{1}{|\eta|^{1/2}}|\xi-\eta|d\eta \| S \|^2_{\q E(t)} \| \tilde v \|_{L^\infty} \lesssim \frac{|\xi|^{3/2}}{t^{1/2}}  \| S \|^2_{\q E(t)} \| \tilde v \|_{L^\infty} \lesssim \frac{1}{t}  \| S \|_{\q E(t)}^2   \| v \|_{\q E(t)}. \]

Now, we further split $\Delta_2 = \Delta_2^1 \cup \Delta_2^2$ where
\[ \Delta_2^1 = \{ \eta \in \m R : |\eta| \ge 10 |\xi|, \ |t^{1/3} \eta| \le 1 \} \quad \text{and} \quad \Delta_2^2 = \{ \eta \in \m R : |\eta| \ge 10 |\xi|, \ |t^{1/3} \eta| \ge 1 \}. \]

(ii) We first focus on $\Delta_2^1$. On $\Delta_2$, we have  $|\xi - \eta| \lesssim |\eta|$, so that
\begin{align*}
|I(\Delta_2^1) | & \lesssim \frac{1}{t^{1/2}} \int_{|\eta| \le t^{-1/3}} \frac{|\xi-\eta|}{|\eta|^{1/2}} d\eta \| S \|^2_{\q E(t)} \| \tilde v \|_{L^\infty} \lesssim \frac{1}{t^{1/2}} \int_{|\eta| \le t^{-1/3}} |\eta|^{1/2} d\eta \| S \|^2_{\q E(t)} \| v \|_{\q E(t)} \\
& \lesssim  \frac{1}{t} \| S \|^2_{\q E(t)} \| v \|_{\q E(t)}.
\end{align*}

(iii) In  $\Delta_2^2$, we perform an integration by parts, to get
\begin{align*}
|I(\Delta_2^2)| & \le  \left|\frac{1}{t^{1/3}} \int_{\Delta_2} e^{it\Psi}\frac{\partial}{\partial \eta}\left(\frac{1}{it\partial_\eta \Psi}K(S,S)(t^{1/3}\eta)(\xi-\eta) \tilde{v}(\xi-\eta) \right)d\eta\right| \\
& \qquad+ \left|\frac{1}{t^{1/3}} \left[ e^{it\Psi}\frac{1}{it\partial_\eta \Psi}K(S,S)(t^{1/3}\eta)\xi_3\tilde{v}(\xi_3)\right]_{\eta \in \partial \Delta_2}\right|.
\end{align*}
Since, on $\Delta_2$, we have  $|\xi - \eta| \lesssim |\eta|$, $|\partial_\eta \Psi| \gtrsim |\eta|^2$, $|\partial_{\eta \eta}^2 \Psi| \lesssim |\eta|$, and  $|\eta| \ge t^{-1/3}$, we can bound the  first term by
\begin{align*}
\MoveEqLeft \frac{1}{t^{1/3}} \int_{|\eta| \ge t^{-1/3}}  \left(  \frac{|\eta|}{t |\eta|^4} \frac{\| S \|_{\q E(t)}^2}{t^{1/6} |\eta|^{1/2}} |\eta| \| \tilde v \|_{L^\infty} + \frac{1}{t |\eta|^2} \frac{\| S \|_{\q E(t)}^2}{t^{1/6} |\eta|^{3/2}} |\eta| \| \tilde v \|_{L^\infty} \right. \\
& \qquad \left.  + \frac{1}{t|\eta|^2} \frac{\| S \|_{\q E(t)}^2}{t^{1/6} |\eta|^{1/2}} |\xi-\eta| | \partial_\eta \tilde v(\xi-\eta)| \right) d\eta \\
& \lesssim \frac{1}{t} \| S \|_{\q E(t)}^2 \|  \| \tilde v \|_{L^\infty} + \frac{1}{t^{4/3}}  \| S \|_{\q E(t)}^2 \left( \int_{|\eta| \ge t^{-1/3}} \frac{d\eta}{|\eta|^3} \right)^{1/2} \frac{1}{t^{1/6}} \| \partial_\xi \tilde v \|_{L^2} \\
& \lesssim \frac{1}{t}  \| S \|_{\q E(t)}^2   \| v \|_{\q E(t)}.
\end{align*}
(on the second line, we used the Cauchy-Schwarz inequality for the last term). For the boundary term, we notice that on $\partial \Delta_2$,  $|\eta| \ge t^{-1/3}$, so that it is bounded by
\[ \frac{1}{t^{1/3}} \left[ \frac{1}{t |\eta|^2} \frac{1}{t^{1/6}|\eta|^{1/2}} \| S \|_{\q E(t)}^2 |\eta| \| \tilde{v} \|_{L^\infty} \right]_{|\eta| = \frac{1}{t^{1/3}}} \lesssim  \frac{1}{t}  \| S \|_{\q E(t)}^2   \| v \|_{\q E(t)}. \]
This proves that $| I(\Delta_2^2) | \lesssim t^{-1} \| S \|_{\q E(t)}^2   \| v \|_{\q E(t)}$.

\bigskip

\emph{Step 2.} We now consider the case $|t^{1/3} \xi | \ge 1$.

Here, we split $\Delta_1 = \Delta_1^1 \cup \Delta_1^2 \cup \Delta_1^3$ where
\begin{align*}
\Delta_1^1 & = \{ \eta \in \m R : |\eta| \le 10 |\xi|, \ |t \xi^3| \ge 1, \min(|\eta - 2\xi|,|\eta-2\xi/3|) \ge |\xi|/10 \}, \\
\Delta_1^2 & = \{ \eta \in \m R : |\eta-2\xi| \le |\xi|/10, \ |t \xi^3| \ge 1\}, \\
\Delta_1^3 & = \{ \eta \in \m R : |\eta-2\xi/3| \le |\xi|/10, \ |t \xi^3| \ge 1\}.
\end{align*}

\bigskip
 
(i) On $\Delta_1^1$, we perform an integration by parts so that
\begin{align}
|I(\Delta_1^1) | & \le  \left|\frac{1}{t^{1/3}}\int_{\Delta_1^1} \eta e^{it\Psi}\frac{\partial}{\partial \eta}\left(\frac{1}{1+it\eta\partial_\eta \Psi}K(S,S)(t^{1/3}\eta)(\xi-\eta)\tilde{v}(\xi-\eta)) \right)d\eta\right| \nonumber \\
& \qquad+ \left|\frac{1}{t^{1/3}} \left[\eta e^{it\Psi}\frac{1}{1+it\eta\partial_\eta \Psi}K(S,S)(t^{1/3}\eta)(\xi-\eta)\tilde{v}(\xi-\eta)\right]_{\eta \in \partial \Delta_1^1}\right| \label{est:I11}
\end{align}
Observe that, on $\Delta_1^1$, $|\partial_\eta \Psi| \gtrsim |\xi|^2$ and $|\partial_{\eta \eta}^2 \Psi| \lesssim |\xi|$. In particular $|\partial_\eta(1+it\eta\partial_\eta \Psi)| \lesssim t |\xi|^2$. Therefore, the integral term is bounded by
\begin{align}
\MoveEqLeft \frac{1}{t^{1/3}} \int_{|\eta| \lesssim |\xi|} \left( \frac{t |\eta \xi^2|}{1 + (t\eta \xi^2)^2} \frac{\| S \|_{\q E(t)}^2}{t^{1/6} |\eta|^{1/2}} |\xi| \| \tilde v\|_{L^\infty}  + \frac{|\eta|}{1+ |t \eta \xi^2|} \frac{\| S \|_{\q E(t)}^2}{t^{1/6} |\eta|^{3/2}}  |\xi| \| \tilde v\|_{L^\infty} \right. \nonumber \\
& \left. + \frac{|\eta|}{1+ |t \eta \xi^2|} \frac{\| S \|_{\q E(t)}^2}{t^{1/6} |\eta|^{1/2}}  \| \tilde v\|_{L^\infty}  \right) d\eta + \frac{1}{t^{1/3}} \int_{|\eta| \lesssim |\xi|} \frac{|\eta|}{1+ |t \eta \xi^2|} \frac{\| S \|_{\q E(t)}^2}{t^{1/6} |\eta|^{1/2}}  |\xi| |\partial_\xi \tilde v(\xi-\eta)| d\eta 
\label{est:I11_int}
\end{align}
Letting $\rho = t \eta \xi^2$ so that $\ds \frac{d\eta}{|\eta|^{1/2}} = \frac{d\rho}{t^{1/2} |\xi| |\rho|^{1/2}}$, (and using again $|\eta| \lesssim |\xi|$), the first three terms are bounded as follows
\begin{align*}
\frac{1}{t^{1/3}} \int_{\m R} \left( \frac{|\rho|}{1+\rho^2} \frac{|\xi|}{t^{1/6}} + \frac{1}{1+|\rho|} \frac{|\xi|}{t^{1/6}} \right) \frac{d\rho}{t^{1/2} |\xi| |\rho|^{1/2}} \| S \|_{\q E(t)}^2 \| \tilde v\|_{L^\infty} \lesssim \frac{1}{t} \| S \|_{\q E(t)}^2 \| v\|_{\q E(t)}.
\end{align*}
We bound the fourth and last terms of \eqref{est:I11_int} by
\begin{align*}  
\MoveEqLeft \frac{1}{t^{1/3}} \left( \int_{|\eta| \lesssim |\xi|} \frac{|\eta \xi^2|}{1+ |t \eta \xi^2|^2} d\eta \right)^{1/2} \| S \|_{\q E(t)}^2 \frac{1}{t^{1/6}} \| \partial_\xi \tilde v \|_{L^2} \\
& \lesssim \frac{1}{t^{1/3}}  \left( \int_{\m R} \frac{1}{t^{1/3}} \frac{|t\eta \xi^2|^{1/3} \xi^2}{1+ |t \eta \xi^2|^2} d\eta \right)^{1/2}  \| S \|_{\q E(t)}^2  \| v \|_{\q E(t)} \\
& \lesssim \frac{1}{t^{1/2}} \left( \int \frac{|\rho|^{1/3} \xi^2}{1+\rho^2} \frac{d\rho}{t \xi^2} \right)^{1/2}  \| S \|_{\q E(t)}^2  \| v \|_{\q E(t)} \lesssim \frac{1}{t} \| S \|_{\q E(t)}^2 \| v\|_{\q E(t)}.
\end{align*}
We now focus on the boundary term in \eqref{est:I11}: on $\partial \Delta_1^1$, $|\eta| \gtrsim |\xi|$, so it can be bounded by
\[ \frac{1}{t^{1/3}} \left[ \frac{|\xi|}{1+t |\xi|^3} \frac{ \| S \|_{\q E(t)}^2}{t^{1/6} |\xi|^{1/2}} |\xi| \| \tilde v \|_{L^\infty} \right] \lesssim \frac{1}{t} \frac{|t\xi^3|^{1/2}}{1+ t\xi^3}  \| S \|_{\q E(t)}^2 \| v \|_{\q E(t)}  \lesssim \frac{1}{t} \| S \|_{\q E(t)}^2 \| v \|_{\q E(t)}.   \]
This proves that $|I(\Delta_1^1)| \lesssim t^{-1} \| S \|_{\q E(t)}^2 \| v \|_{\q E(t)}$.

\bigskip

(ii) For $\Delta_1^2$, we are near the stationary point $2\xi$. Denote $q = \eta - 2\xi$. An integration by parts yields
\begin{align}
|I(\Delta_1^2)| &\lesssim \left|\frac{1}{t^{1/3}}\int_{\Delta_1^2} q e^{it\Psi}\frac{\partial}{\partial \eta}\left(\frac{1}{1+itq\partial_\eta \Psi}K(S,S)(t^{1/3}\eta)(\xi-\eta)\tilde{v}(\xi-\eta) \right)d\eta\right|  \label{est:I12} \\
& \qquad+ \left|\frac{1}{t^{1/3}} \left[(\eta - 2\xi) e^{it\Psi}\frac{1}{1+it (\eta-2\xi)\partial_\eta \Psi}K(S,S)(t^{1/3}\eta) (\xi-\eta) \tilde{v}(\xi-\eta)\right]_{\eta \in \partial \Delta_1^2}\right|
\end{align}
On $\Delta_1^2$, $|q\xi| \lesssim |\partial_\eta \Psi| \lesssim |q \xi|$, $|\partial_{\eta \eta}^2 \Psi| \lesssim |\xi|$, $|\xi - \eta| \lesssim |\xi|$, $|\xi| \lesssim |\eta| \lesssim |\xi|$. Hence, we can bound the integral term in \eqref{est:I12} by
\begin{align*}
\MoveEqLeft \frac{1}{t^{1/3}}\int_{\Delta_1^2} \left( \frac{t q^2 |\xi|}{1+ (tq^2|\xi|)^2} \frac{\| S \|_{\q E(t)}^2 }{t^{1/6} |\xi|^{1/2}} |\xi| \| \tilde v \|_{L^\infty} +  \frac{q}{1+ |tq^2 \xi|} \frac{\| S \|_{\q E(t)}^2}{t^{1/6} |\xi|^{3/2}}  |\xi| \| \tilde v \|_{L^\infty} \right.  \\
& \qquad \left. + \frac{|q|}{1+ |tq^2 \xi|} \frac{\| S \|_{\q E(t)}^2}{t^{1/6} |\xi|^{1/2}}  |\xi| |\partial_\xi \tilde v(\xi -\eta)| \right) d\eta
\end{align*}
Let here $\ds \rho = t q^2 |\xi|$, $\ds d\eta = \frac{d\rho}{2|t \xi \rho |^{1/2}}$  so that the first term is bounded by
\[  \frac{1}{t^{1/3}}\int_{\m R} \frac{\rho}{1+\rho^2} \frac{|\xi|^{1/2}}{t^{1/6}} \frac{d\rho}{|t \xi \rho|^{1/2}}\| S \|_{\q E(t)}^2 \| \tilde v \|_{L^\infty}  \lesssim \frac{1}{t} \| S \|_{\q E(t)}^2 \| \tilde v \|_{\q E(t)}. \]
For the second term, notice that $\ds  \frac{q}{1+ |tq^2 \xi|} \le \frac{1}{|t \xi|^{1/2}}$, and as the measure $|\Delta_1^2| \lesssim |\xi|$, it is bounded by
\[  \frac{|\Delta_1^2|}{t^{1/3}}  \frac{1}{|t \xi|^{1/2}}  \frac{1}{t^{1/6} |\xi|^{1/2}} \| S \|_{\q E(t)}^2 \| \tilde v \|_{L^\infty} \lesssim \frac{1}{t} \| S \|_{\q E(t)}^2 \| \tilde v \|_{\q E(t)}. \]
For the third term, using the Cauchy-Schwarz inequality and then the fact that  on $\Delta_1^2$, $|q|\lesssim |\xi|$,
\begin{align*}
\MoveEqLeft \frac{1}{t^{1/3}}  \frac{\| S \|_{\q E(t)}^2}{t^{1/6}}  |\xi|^{1/2}  \left( \int_{\Delta_1^2} \frac{q^2}{1+(tq^2 \xi)^2} d\eta \right)^{1/2} \| \partial_\xi \tilde v \|_{L^2} \\
& \lesssim \frac{|\xi|^{1/2}}{t^{1/3}}  \left( \int_{\m R} \frac{|q|^{5/3} |\xi|^{1/3}}{1+(tq^2 \xi)^2} d\eta \right)^{1/2}  \| S \|_{\q E(t)}^2 \| v \|_{\q E(t)} \\
& \lesssim \frac{|\xi|^{1/2}}{t^{1/3}}  \left( \int_{\m R} \frac{1}{t^{5/6} |\xi|^{1/2}}  \frac{|\rho|^{5/6}}{1+|\rho|^2} \frac{d\rho}{|t \xi \rho|^{1/2}} \right)^{1/2}  \| S \|_{\q E(t)}^2 \| v \|_{\q E(t)} \lesssim \frac{1}{t} \| S \|_{\q E(t)}^2 \|  v \|_{\q E(t)}. 
\end{align*}
We observe that on the last line, the $\xi$ cancel out and the $\rho$ integral is convergent, leaving the desired power of $t$.

For the boundary term in \eqref{est:I12}, notice that on $\partial \Delta_1^2$, $\eta$, $\xi-\eta$ and $\eta - 2\xi$ are of the order of $\xi$, and $|\partial_\eta \Psi| \gtrsim |\xi|^2$, so that we can bound it by
\[ \frac{1}{t^{1/3}} |\xi| \frac{1}{1+ t|\xi|^3} \frac{\| S \|_{\q E(t)}^2}{t^{1/6}|\xi|^{1/2}} |\xi| \| \tilde v\|_{L^ \infty} \lesssim \frac{1}{t^{1/2}} \frac{|\xi|^{3/2}}{1+t|\xi|^3} \| S \|_{\q E(t)}^2 \| \tilde v\|_{L^ \infty} \lesssim \frac{1}{t} \| S \|_{\q E(t)}^2 \| \tilde v\|_{L^ \infty}, \]
because $\ds \frac{|\xi^{3/2}}{ 1 + t |\xi|^3} \le \frac{1}{t^{1/2}}$. We conclude that $\ds |I(\Delta_1^2)| \lesssim t^{-1} \| S \|_{\q E(t)}^2 \| v \|_{\q E(t)}$.

\bigskip

(iii) On $\Delta_1^3$, we are near the stationary point $2\xi/3$: letting $q = \eta - 2\xi/3$, completely similar computations as in the $\Delta_1^2$ case give the same bound:
\[ |I(\Delta_1^3)| \lesssim  t^{-1} \| S \|_{\q E(t)}^2   \| v \|_{\q E(t)}. \]

\bigskip

(iv) On $\Delta_2$,  $|\eta| \ge 10 |\xi| \ge 10 t^{-1/3}$, so that we can repeat the same computations as for $\Delta_2^2$ in Step 1:
\begin{align*}
|I(\Delta_2) | & \lesssim  \frac{1}{t} \| S \|^2_{\q E(t)} \| v \|_{\q E(t)}.
\end{align*}

\bigskip

Summing up, we conclude that $| I | \lesssim t^{-1} \| S \|_{\q E(t)}^2   \| v \|_{\q E(t)}$, as claimed.
\end{proof}

\begin{lem}\label{lem:estE}
	Given $t\in(0,1)$, $z\in \mathcal{Z}$ and $v\in \mathcal{E}(0,1)$, 
\begin{align*}
|v(t,x)| & \lesssim \frac{1}{t^{1/3}\jap{x/t^{1/3}}^{1/4}} \|v(t)\|_{\mathcal{E}(t)}, & \quad |e^{-t \partial_x^3} z(x)| & \lesssim \min\left(1,\frac{1}{t^{1/3}\jap{x/t^{1/3}}^{1/4}} \right) \vvvert z\vvvert, \\
|\partial_x v(t,x)| & \lesssim \frac{1}{t^{2/3}} \jap{x/t^{1/3}}^{1/4} \|v(t)\|_{\mathcal{E}(t)}, & \quad |\partial_x e^{-t \partial_x^3} z(x)| & \lesssim  \vvvert z\vvvert, 
\end{align*}
In particular, $u=v+e^{-t\partial_x^3}z$ satisfies
	\begin{align*}
	\|uu_x(t)\|_{L^\infty} & \lesssim \frac{1}{t} \|v(t)\|_{\mathcal{E}(t)} (\|v(t)\|_{\mathcal{E}(t)} + \nz{z})  + \vvvert z\vvvert^2, \\
	\|u(t)\|_{L^\infty} & \lesssim \frac{1}{t^{1/3}} \|v(t)\|_{\mathcal{E}(t)} + \vvvert z\vvvert, \\
	\text{and} \qquad 
	\|u(t)\|_{L^6} & \lesssim \frac{1}{t^{5/18}} \|v(t)\|_{\mathcal{E}(t)} + \vvvert z\vvvert.
	\end{align*}
\end{lem}

In the above bounds on $u$, observe the gain in powers of $t$ on the $z$ terms.

\begin{proof}
	Since $\jap{\xi} \hat z \in L^1(\R)$, one has directly $e^{-t\partial_x^3}z\in W^{1,\infty}(\R)$. Also, $\hat z \in L^2(\m R)$, so that $e^{-t\partial_x^3}z\in L^2(\m R)$. In particular,
\[ \| e^{-t\partial_x^3}z \|_{L^6} \lesssim \vvvert z\vvvert. \]
Furthermore, by \cite[Lemma 6]{CCV20},
	\[
	|v(t,x)|\lesssim \frac{1}{t^{1/3}\jap{x/t^{1/3}}^{1/4}}\|v(t)\|_{\mathcal{E}(t)},\quad |v_x(t,x)|\lesssim \frac{1}{t^{2/3}}\jap{x/t^{1/3}}^{1/4}\|v(t)\|_{\mathcal{E}(t)}.
	\]
	It remains to prove the pointwise spatial decay for 
	\[
	(e^{-t\partial_x^3}z)(x)=\frac{1}{t^{1/3}}\int \Ai\left(\frac{x-y}{t^{1/3}}\right)z(y)dy.
	\]
	If $|x-y|\le |x|$, then
	\begin{align*}
	|(e^{-t\partial_x^3}z)(x)|&\lesssim \left(\frac{1}{t^{1/3}}\int_{|x-y|\le |x|} \Big\langle\frac{x-y}{t^{1/3}}\Big\rangle^{-1/4}\jap{x}^{-1}dy\right)\|\jap{x}z\|_{L^\infty}\\&\lesssim \left(\frac{|x|}{t^{1/3}}\right)^{3/4}\jap{x}^{-1}\nz{z}\lesssim \frac{1}{t^{1/3}\jap{x/t^{1/3}}^{1/4}}\nz{z}.
	\end{align*}
	If $|x-y| \ge |x|$, then
	\[
	|(e^{-t\partial_x^3}z)(x)|\lesssim \frac{1}{t^{1/3}}\int \Big\langle\frac{x-y}{t^{1/3}}\Big\rangle^{-1/4}|z(y)|dy\lesssim \frac{1}{t^{1/3}\jap{x/t^{1/3}}^{1/4}}\|z\|_{L^1}. 
	\]
	For the $uu_x$ estimate, one notices that
	\[ uu_x = vv_x + ve^{-t\partial_x^3}z_x + v_x e^{-t\partial_x^3}z + (e^{-t\partial_x^3}z)(e^{-t\partial_x^3}z_x). \]
	The first three terms are estimated by observing the cancellation of the $\jap{x/t^{1/3}}^{1/4}$ factor, and the quadratic term in $z$ is bounded by a constant. The other estimates on $u$ are easy consequences.
\end{proof}

\section{Construction of an approximation sequence}\label{sec:aprox}

Now that we possess the necessary estimates, let us begin the construction of the error function $w$ in the same spirit as \cite{CCV20}. Fix $\chi\in\mathcal{S}(\R)$ such that $0<\chi\le 1$ and $\chi\equiv 1$ on $[-1,1]$ and let $\chi_n(\xi)=\chi^2(\xi/n)$.
Given $u\in \mathcal{S}'(\R)$, set
\[
\widehat{\Pi_n u} := \chi_n\hat{u}
\]
and define the approximation space at time $t>0$
\[ X_n(t) := \left\{ u\in \mathcal{S}'(\R): \| u \|_{X_n(t)} < \infty \right\},  \]
where
 \begin{equation}
\|u\|_{X_n(t)} :=\left\|  e^{it\xi^3}\hat u \chi_n^{-1}\right\|_{L^\infty} + \left\| \partial_\xi \left(e^{it\xi^3}\hat{u} \right) \chi_n^{-1/2} \right\|_{L^2}.
\end{equation}
If $I \subset [0,+\infty)$ is an interval,
\[ \| u \|_{X_n(I)} := \sup_{t \in I} \| u(t) \|_{X_n(t)} = \sup_{t \in I} \| \tilde u(t) \chi_n^{-1}\|_{L^\infty}  + \| \partial_\xi \tilde u(t) \chi_n^{-1/2} \|_{L^2}, \]
and 
\[  X_n(I) := \left\{ u\in \q  C(I,\mathcal{S}'(\R )): \tilde u \chi_n^{-1} \in \q  C(I, \q C_b(\R)),\  \partial_\xi \tilde u \chi_n^{-1/2} \in \q  C(I, L^2(\R))  \right\}. \]

Next, we consider a suitable approximation of the error $w$ by cutting off the nonlinearity at large frequencies. However, one must also truncate the self-similar solution to avoid problems with the linear term $(SSw)_x$. In order to keep the self-similar structure, we set
\[
\tilde{S}_n(t,\xi)=\chi_n(t^{1/3} \xi)\tilde{S}(t^{1/3}\xi).
\]
This cut-off is well-behaved in $\mathcal{E}$:
\begin{align*}
\|S_n\|_{\mathcal{E}((0,+\infty))} & = \| S_n \|_{\mathcal{E}(1)} = \| \tilde \chi_n \tilde S \|_{L^\infty} + \| \partial_{\xi} (\chi_n \tilde S) \|_{L^2}  \\
& \lesssim \| \tilde S \|_{L^\infty} + \| \partial_x \tilde S \|_{L^2} + \| \partial_x \chi_n \|_{L^2} \| \tilde S \|_{L^\infty} \\
& \lesssim \| S \|_{\q E(1)}.
\end{align*}
The resulting approximate problem is
\[ \partial_t u_n + \partial^3_x u_n = \Pi_n \partial_x(u_n^3), \quad u_n(t) - S_n(t) \to z \quad \text{as } t \to 0^+. \]
Writing
\[ u_n(t)= : S_n(t) + e^{-t\partial_x^3}z + w_n(t), \]
one arrives at the (equivalent) problem for the error $w_n$
\begin{gather}\label{eq:w_n}
\partial_tw_n +  \partial^3_x w_n = \Pi_n \partial_x(u_n^3-S_n^3), \quad w_n(0)=0.
\end{gather}
Observe that, in frequency space, this equation reads as
\[
\partial_t \tilde w_n = \chi_n (N[\tilde u_n] - N[\tilde S_n]).
\]
\begin{prop}\label{lem:existwn}
	Given $z\in \mathcal{Z}$ and a self-similar solution $S\in \mathcal{E}((0,+\infty))$, there exists $T_n=T_n(z,S)>0$ and a unique maximal solution $w_n\in X_n([0,T_n))$ of \eqref{eq:w_n}, in the sense that if $T_n < +\infty$, then $\| w_n \|_{X_n(t)} \to +\infty$ as $t \to T_n^-$. Moreover, there exists $0 <T^0_n<T_n$ such that
	\[
	\forall t \in [0,T_n^0], \quad \|w_n\|_{X_n(t)}\lesssim_n t^{1/3}.
	\]
\end{prop}

\begin{nb}
	Since $w_n(0)=0$, $\tilde{w}_n$ will not have any jump discontinuity at $\xi=0$. Therefore $\partial_\xi \tilde{w}_n$ will be bounded in $L^2(\R)$ (see Proposition \ref{prop:aprioriI}). 
\end{nb}

\begin{proof}[Sketch of the proof.]
	The proof follows from a fixed-point argument in $X_n([0,T])$ for the map $\Psi$ defined as
	\[
	\widetilde{\Psi[w_n]}(t,\xi)=\chi_n(\xi) \int_0^t (N[\tilde u_n] - N[\tilde S_n])(s,\xi) ds.
	\]
	for some $0 < T \le 1$ to be determined later. Let us consider the source term 
\[
	N[\tilde S_n, \tilde S_n,\tilde z] (t,\xi)= i \frac{\xi}{4 \pi^2} \iint_{\xi_1+\xi_2+\xi_3=\xi} e^{it\Phi}\tilde{S}_n(t,\xi_1)\tilde{S}_n(t,\xi_2)\hat{z}(\xi_3)d\xi_1 d\xi_2.
	\]
The idea is to place the $z$ factor in $L^1$ based spaces. To bound the $L^\infty$ term in the $X_n$ norm, we estimate
	\begin{align*}
	| N[\tilde S_n, \tilde S_n,\tilde z] (t,\xi)| &\lesssim \| \tilde S \|_{L^\infty}^2 \iint_{\xi_1+\xi_2+\xi_3=\xi} (|\xi_1| + |\xi_2| + |\xi_3|) \chi_n(t^{1/3}\xi_1)\chi_n(t^{1/3}\xi_2)|\hat z(\xi_3)|d\xi_1 d\xi_2\\
	&\lesssim \frac{\| \tilde S \|_{L^\infty}^2}{t^{2/3}} \|\xi\chi_n\|_{L^\infty} \|\chi_n\|_{L^1}\|\hat z\|_{L^1} + \frac{\| \tilde S \|_{L^\infty}^2}{t^{1/3}} \|\chi_n\|_{L^\infty} \|\chi_n\|_{L^1}\| \jap{\xi}\hat z\|_{L^1}  \\
	& \lesssim \frac{1}{t^{2/3}} \| \jap{\xi} \chi_n \|_{L^\infty}  \|\chi_n\|_{L^1} \| \jap{\xi}\hat z\|_{L^1}  \| \tilde S \|_{L^\infty}^2.
	\end{align*}
	Therefore, for $t \in [0,1]$,
	\begin{align} 
	\left\| \int_0^t N[\tilde S_n,\tilde S_n,\tilde z](s,\xi) ds  \right\|_{L^\infty} & \lesssim  \int_0^t\frac{\| \tilde S \|_{L^\infty}^2}{s^{2/3}}  \| \jap{\xi} \chi_n \|_{L^\infty}  \|\chi_n\|_{L^1} \| \jap{\xi}\hat z\|_{L^1} ds \nonumber \\
	& \lesssim t^{1/3}  \| \jap{\xi} \chi_n \|_{L^\infty}  \|\chi_n\|_{L^1} \| \jap{\xi}\hat z\|_{L^1} \| \tilde S \|_{L^\infty}^2. \label{est:Xn_N[SSz]_1}
	\end{align}
	For the estimate of the $L^2$ term in $X_n$, we have to bound in $L^2$
\[ \chi_n^{-1/2}(\xi) \partial_\xi \left( \chi_n(\xi)  \int_0^t (N[\tilde S_n,\tilde S_n, \tilde z])(s,\xi) ds \right). \]
This requires us to consider the following three terms:
	\begin{align*}
	\MoveEqLeft \left|\chi_n^{-1/2}\partial_\xi(\xi\chi_n(\xi)) N[\tilde S_n, \tilde S_n,\tilde z] (t,\xi) \right| \\ 
& \lesssim |\partial_\xi(\xi\chi_n^{1/2}(\xi))| \| \tilde S \|_{L^\infty}^2 \iint_{\xi_1+\xi_2+\xi_3=\xi} \chi_n(t^{1/3}\xi_1)\chi_n(t^{1/3}\xi_2)|\hat{z}(\xi_3)|d\xi_1 d\xi_2 \\
& \lesssim |\partial_\xi(\xi\chi_n^{1/2}(\xi))|\frac{1}{t^{1/3}}\|\chi_n\|_{L^1}\|\hat{z}\|_{L^1}\|\chi_n\|_{L^\infty} \| \tilde S \|_{L^\infty}^2,\\
	\MoveEqLeft \left|\xi\chi_n^{1/2}(\xi)\iint_{\xi_1+\xi_2+\xi_3=\xi} t(\partial_\xi\Phi)\  e^{it\Phi}\tilde{S}_n(t,\xi_1)\tilde{S}_n(t,\xi_2)\tilde{z}(t,\xi_3)d\xi_1 d\xi_2\right| \\
& \lesssim |\xi\chi_n^{1/2} (\xi)| \| \tilde S \|_{L^\infty}^2  \iint_{\xi_1+\xi_2+\xi_3=\xi} t(\xi_1^2+\xi_2^2+\xi_3^2)\chi_n(t^{1/3}\xi_1)\chi_n(t^{1/3}\xi_2)|\hat{z}(\xi_3)|d\xi_1 d\xi_2 \\
& \lesssim |\xi\chi_n^{1/2}(\xi)|\|\jap{\xi}^2\chi_n\|_{L^1}\|\jap{\xi}^2\hat{z}\|_{L^1}\|\jap{\xi}^2\chi_n\|_{L^\infty} \| \tilde S \|_{L^\infty}^2, \\
	\MoveEqLeft \left|\xi\chi_n^{1/2}(\xi)N[\tilde S_n, \tilde S_n, \partial_\xi \tilde z] (t,\xi) \right| \\ 
& \lesssim |\xi\chi_n^{1/2}(\xi)|  \| \tilde S \|_{L^\infty}^2  \iint_{\xi_1+\xi_2+\xi_3=\xi} \chi_n(t^{1/3}\xi_1)\chi_n(t^{1/3}\xi_2)|\partial_\xi\hat{z}(\xi_3)|d\xi_1 d\xi_2 \\
& \lesssim |\xi\chi_n^{1/2}(\xi)| \frac{1}{t^{1/3}} \|\chi_n\|_{L^1}\|\hat{z}\|_{W^{1,1}}\|\chi_n\|_{L^\infty} \| \tilde S \|_{L^\infty}^2.
	\end{align*}
	Since $\chi_n^{1/2}\in \mathcal{S}(\R)$, these terms are bounded in $L^2$ and
\[ \left\| \chi_n^{-1/2}(\xi) \partial_\xi \left( \chi_n N[\tilde S_n, \tilde S_n,\tilde z] (s,\xi) \right) \right\|_{L^2} \lesssim t^{-1/3} \| \tilde S \|_{L^\infty}^2. \]
After integration in time, we get for $t \in [0,1]$
\begin{equation}
 \left\| \chi_n^{-1/2}(\xi) \partial_\xi \left( \chi_n \int_0^t N[\tilde S_n, \tilde S_n,\tilde z] (s,\xi) ds \right)  \right\|_{L^2} \lesssim t^{2/3} \| \tilde S \|_{L^\infty}^2,  \label{est:Xn_N[SSz]_2} 
\end{equation}
and so, together with \eqref{est:Xn_N[SSz]_1}, we conclude
\[ \left\| \chi_n \int_0^t N[\tilde S_n, \tilde S_n, \tilde z](s) ds \right\|_{X_{n}(t)} \lesssim t^{1/3} \| \tilde S \|_{L^\infty}^2. \]
The others source terms (where $z$ is quadratic or cubic) can be treated in a similar fashion (and are actually better behaved).

Similarly, let us consider the linear term $N[\tilde S_n, \tilde S_n, \tilde w_n]$:
\begin{align*}
\MoveEqLeft | N[\tilde S_n, \tilde S_n,\tilde w_n] (t,\xi)| \\
&\lesssim \| \tilde S \|_{L^\infty}^2 \| \tilde w_n (t) \chi_n^{-1} \|_{L^\infty} \iint (|\xi_1| + |\xi_2| + |\xi_3|) \chi_n(t^{1/3}\xi_1)\chi_n(t^{1/3}\xi_2) ) \chi_n(\xi_3) d\xi_1 d\xi_2\\
&\lesssim \| \tilde S \|_{L^\infty}^2 \| w_n \|_{X_n(t)} \left( \frac{1}{t^{2/3}} \|\xi\chi_n\|_{L^\infty} \|\chi_n\|_{L^1}^2 + \frac{1}{t^{1/3}} \|\chi_n\|_{L^1} \| \chi_n \|_{L^\infty}  \| \xi \chi_n  \|_{L^1}  \right) \\
& \lesssim \frac{1}{t^{2/3}} \| \jap{\xi} \chi_n \|_{L^\infty}  \|\chi_n\|_{L^1}  \| \tilde S \|_{L^\infty}^2 \| w_n \|_{X_n(t)}.
\end{align*}
Hence, after integration in time,
\begin{align} \label{est:NSSw_Xn_1}
\left\|  \int_0^t N[\tilde S_n, \tilde S_n, \tilde w_n](s,\xi) ds \right\|_{L^\infty} \lesssim t^{1/3}  \| w_n \|_{X_n([0,t])}.
\end{align}
For the $L^2$ term, we have
\begin{align*}
	\MoveEqLeft \left|\chi_n^{-1/2}\partial_\xi(\xi\chi_n(\xi)) N[\tilde S_n, \tilde S_n,\tilde w_n] (t,\xi) \right| \\ 
& \lesssim  |\partial_\xi(\xi\chi_n^{1/2}(\xi))| \| \tilde S \|_{L^\infty}^2 \| \tilde w_n (t) \chi_n^{-1} \|_{L^\infty} \iint \chi_n(t^{1/3}\xi_1)\chi_n(t^{1/3}\xi_2)\chi_n(\xi_3)d\xi_1 d\xi_2 \\
& \lesssim |\partial_\xi(\xi\chi_n^{1/2}(\xi))|\frac{1}{t^{1/3}} \|\chi_n\|_{L^\infty} \|\chi_n\|_{L^1}^2  \| \tilde S \|_{L^\infty}^2 \| w_n \|_{X_n(t)},\\
	\MoveEqLeft \left|\xi\chi_n^{1/2}(\xi)\iint_{\xi_1+\xi_2+\xi_3=\xi} t(\partial_\xi\Phi)\  e^{it\Phi}\tilde{S}_n(t,\xi_1)\tilde{S}_n(t,\xi_2)\tilde w_n(t,\xi_3)d\xi_1 d\xi_2\right| \\
& \lesssim \| \tilde S \|_{L^\infty}^2 \| \tilde w_n (t) \chi_n^{-1} \|_{L^\infty}  |\xi\chi_n^{1/2}(\xi)|\iint t(\xi_1^2+\xi_2^2+\xi_3^2)\chi_n(t^{1/3}\xi_1)\chi_n(t^{1/3}\xi_2) \chi(\xi_3) d\xi_1 d\xi_2 \\
& \lesssim |\xi\chi_n^{1/2}(\xi)|\|\jap{\xi}^2\chi_n\|_{L^\infty }\|\chi_n \|_{L^1}^2 \| \tilde S \|_{L^\infty}^2  \| w_n \|_{X_n(t)}, \\
	\MoveEqLeft \left|\xi\chi_n^{1/2}(\xi)N[\tilde S_n, \tilde S_n, \partial_\xi \tilde w_n] (t,\xi) \right| \\ 
& \lesssim   |\xi\chi_n^{1/2}(\xi)|  \| \tilde S \|_{L^\infty}^2 \iint \chi_n(t^{1/3}\xi_1)\chi_n(t^{1/3}\xi_2)|\partial_\xi \tilde w_n(t,\xi_3)|d\xi_1 d\xi_2 \\
& \lesssim |\xi\chi_n^{1/2}(\xi)| \| \tilde S \|_{L^\infty}^2 \| \partial_\xi \tilde w_n  \chi_n^{-1/2} \|_{L^2} \left( \iint \chi_n^2(t^{1/3}\xi_1) \chi_n(\xi_3) d\xi_1  \right)^{1/2}\chi_n (t^{1/3}\xi_2)d\xi_2 \\
& \lesssim  |\xi\chi_n^{1/2}(\xi)| \| \tilde S \|_{L^\infty}^2 \| w_n \|_{X_n(t)} \frac{1}{t^{1/3}} \| \chi_n \|_{L^2} \| \| \chi_n \|_{L^\infty} \| \chi_n \|_{L^1}.
\end{align*}
Therefore, taking the $L^2$ norm in $\xi$ gives
\[ \left\| \chi_n^{-1/2} \partial_\xi \left( \chi_n N[\tilde S_n, \tilde S_n,\tilde w_n] (t) \right) \right\|_{L^2} \lesssim t^{-1/3} \| \tilde S \|_{L^\infty}^2 \| w_n \|_{X_n(t)}. \]
Integrating in time, we get, for $t \in [0,1]$,
\begin{equation}
 \left\| \chi_n^{-1/2} \partial_\xi \left( \chi_n \int_0^t N[\tilde S_n, \tilde S_n,\tilde z] (s) ds \right)  \right\|_{L^2} \lesssim t^{2/3} \| \tilde S \|_{L^\infty}^2 \| w_n \|_{X_n([0,t])},
\end{equation}
and with \eqref{est:NSSw_Xn_1}, we obtain
\[ \left\| \chi_n \int_0^t N[\tilde S_n, \tilde S_n, \tilde w_n](s,\xi) ds \right\|_{X_n(t)} \lesssim t^{1/3}  \| \tilde S \|_{L^\infty}^2 \| w_n \|_{X_n([0,t])}. \]
All the remaining terms can be estimated similarly (and enjoy in fact better bounds), and the difference estimates can be performed in the same way as well. Choosing $T$ sufficently small, $\Phi$ becomes a contraction over $X_n([0,T])$. This concludes the proof.
\end{proof}

In order to take the limit $n\to \infty$, one must prove that the maximal existence time $T_n$ does not tend to $0$ and also that the approximations remain bounded in $\mathcal{E}$. To do this, we  actually prove \emph{a priori} bounds in the stronger spaces $X_n$, thus tackling both problems at once. The methodology to prove this follows the heuristics presented in the introduction.

In the remainder of this work, we now assume that, for some small $\delta$ to be fixed later on, \eqref{eq:small} holds:
\[	\nz{z}+\|S\|_{\mathcal{E}((0,+\infty))}<\delta. \]
In order to bound $\partial_\xi\tilde{w}_n$, the scaling operator $I$ comes into play: we recall its definition given in \eqref{def:I} 
\[
\widehat{Iu}(t,\xi) = ie^{it\xi^3}\left(\partial_\xi\tilde{u} - \frac{3t}{\xi}\partial_t\tilde{u}\right).
\]

\begin{prop} \label{prop:aprioriI}
There exist $\delta_0>0$ such that, given $\delta \in (0,\delta_0]$, the solution $w_n$ of Proposition \ref{lem:existwn} satisfies $T_n >1$ and
\begin{equation}
\label{est:wn_E_L2}
\forall t \in (0,1], \quad \| w_n(t) \|_{\q E(t)} \lesssim \delta^3 t^{1/9} \quad \text{and} \quad \| w_n(t) \|_{L^2} \lesssim \delta^3 t^{1/18}.
\end{equation}
Moreover, $\partial_\xi \tilde w_n(t) \in L^2(\m R)$, $\partial_t \tilde w_n(t) \in L^\infty(\m R)$ and 
\begin{gather}
\label{est:wn_xL2}
 \forall t \in (0,1], \quad \| \partial_\xi \tilde w_n(t) \|_{L^2} \lesssim \delta^3 t^{5/18}  \quad \text{and} \quad \left\| \partial_t \tilde{w}_n(t) \right\|_{L^\infty} \lesssim \frac{\delta^3}{t^{8/9}}.
 \end{gather}
\end{prop}

\begin{proof}
We perform a bootstrap argument. Fix $\delta_0 \in (0,1)$ and $A = A(\delta_0) \ge 1 $ to be chosen later on. For now, we only require $A\delta_0^2 \le 1$. We let $\delta \in (0,\delta_0]$. 

Define, for any $t\in (0,T_n)$,
	\[
	f_n(t)=t^{-1/9} \left( \|\tilde{w}_n(t)\chi_n^{-1}\|_{L^\infty} + t^{-1/6} \|\partial_\xi \tilde{w}_n(t) \chi_n^{-1/2}\|_{L^2} \right)
	\]
	and
	\[
	\tau_n=\sup \left\{ t \in [0, \min(1,T_n)) : \forall s \in (0,t], \  f_n(s)\le A \delta^3 \right\}.
	\]
In the following argument, the implicit constants in $\lesssim$ do not depend on $\delta$ or $A$.
	
From Proposition \ref{lem:existwn}, $f_n$ is continuous on $(0,T_n)$ and $f_n(t) \lesssim t^{1/18}$ for $t \in [0,T_n^0]$, so that $f_n(t) \to 0$ as $t \to 0^+$.
In particular, $\tau_n >0$.

\bigskip

\emph{Step 1. Improved estimates on the nonlinear term.}  

First observe that	
\begin{equation} \label{est:boot_w_E}
\forall t \in (0,\tau_n), \quad \|w_n(t)\|_{\mathcal{E}(t)} \le t^{1/9}f_n(t)\le A \delta^3 t^{1/9}.
\end{equation}

Using the estimates from Lemma \ref{lem:estE}, for all $t \in [0,\tau_n]$,
\begin{align*}
\|u_n^3-S_n^3\|_{L^2} & \lesssim \left(\|w_n\|_{L^6}^2 + \|S_n\|_{L^6}^2 + \|e^{-t\partial_x^3}z\|_{L^6}^2\right)\left(\|w_n\|_{L^6}+ \|e^{-t\partial_x^3}z\|_{L^6}\right) \\
& \lesssim \left( \frac{1}{t^{5/9}} \|w_n(t)\|_{\mathcal{E}(t)}^2 + \frac{1}{t^{5/9}} \|S_n(t)\|_{\mathcal{E}(t)}^2 + \nz{z}^2 \right) \left(  \frac{1}{t^{5/18}} \|w_n(t)\|_{\mathcal{E}(t)} +  \nz{z} \right) \\
& \lesssim  \frac{1}{t^{5/9}} \left( (A\delta^3)^2 t^{2/9} + \delta^2 \right) \left( \frac{A\delta^3 t^{1/9}}{t^{5/18}} + \delta \right).
\end{align*}
Recall that $A\delta^2 \le 1$ so that $A\delta^3 t^{1/9} \le \delta$, and so
\begin{equation} \label{est:boot_NL}
\forall t \in [0,\tau_n),  \quad \|u_n^3-S_n^3\|_{L^2}  \lesssim \frac{A\delta^5}{t^{13/18}} + \frac{\delta^3}{t^{5/9}}.
\end{equation}
Let emphasize the gain of $t^{1/9}$ with respect to the rough estimate of Lemma \ref{lem:estE}
\[ \| v^3 \|_{L^2} = \| v \|_{L^6}^3 \lesssim \frac{\| v \|_{\q E(t)}^3}{t^{5/6}}. \]

\bigskip

\emph{Step 2. A priori estimate for $Iw_n$.} In this step, we prove that
\begin{equation}\label{lem:aprioriI}
\forall t \in [0,\tau_n), \quad  \left\| \widehat{Iw_n}(t) \chi_n^{-1/2}\right\|_{L^2} \lesssim \delta^3 t^{5/18} \left( t^{1/18} +  A\delta^2 \right).  
\end{equation}

Let us notice that from the equation for $\tilde{w}_n$ and Lemma \ref{lem:estE}, given $t<T_n^0$,
\begin{align*}
\left\|\widehat{Iw_n}(t)\chi_n^{-1/2}\right\|_{L^2} & \le \|\partial_\xi \tilde{w}_n(t)\chi_n^{-1/2}\|_{L^2} + 3t\|\chi_n^{1/2} \q F(u_n^3-S_n^3)\|_{L^2}\\
&\lesssim \|\partial_\xi \tilde{w}_n(t)\chi_n^{-1/2}\|_{L^2} + 3t\| u_n^3-S_n^3 \|_{L^2} \\
& \lesssim \|w_n(t)\|_{X_n(t)} + t^{1/6}(\delta^3+\|w_n(t)\|_{X_n(t)}^3)\lesssim t^{1/3} + t^{1/6}(\delta^3+t).
\end{align*}
We conclude that
\begin{equation}\label{est:Iwn_init}
\forall \gamma <1/6, \quad \left\|\widehat{Iw_n}(t)\chi_n^{-1/2}\right\|_{L^2} = o(t^\gamma) \quad \text{as} \quad t \to 0^+.
\end{equation}

Now that we have a control near $t=0$, we aim to control $Iw_n$ for $0<t<1$. Denote
\[\widehat{\Pi'_n u}(\xi)=\chi_n'(\xi)\hat{u}(\xi).\]
A simple computation yields
	\begin{align*}
	(\partial_t + \partial^3_x)Iw_n &= 3\Pi_n(u_n^2(Iu_n)_x-S_n^2(IS_n)_x) + \Pi_n'(u_n^3-S_n^3)_x \\&= 3\Pi_n(u_n^2(Iw_n)_x)+ 3\Pi_n(u_n^2(Ie^{-t\partial_x^3}z)_x) + \Pi_n'(u_n^3-S_n^3)_x
	\end{align*}
	where we used the decisive property $(IS_n)_x\equiv 0$. Equivalently, on Fourier side it writes
\begin{align*}
\partial_t \widehat{I w_n}  - i \xi^3 \widehat{I w_n} &= 3 \chi_n \q F\left(u_n^2(Iw_n)_x\right) + 3 \chi_n \q F\left(u_n^2(Ie^{-t\partial_x^3}z)_x\right) +  i \xi \chi_n' \q F\left(u_n^3-S_n^3\right)
\end{align*}
We now multiply by $\overline{\widehat{I w_n}} \chi_n^{-1}$, integrate on $\m R$ and take the real part. Due to the jump of $\tilde u_n$ at $\xi=0$, an extra care should be taken in the computations: one should first integrate over $\m R\setminus (-\varepsilon,\varepsilon)$ and let $\e \to 0$. This procedure shows that no unexpected term occur, we refer to  \cite[Lemma 12]{CCV20} for full details. This justifies the following computations:
\begin{align*}
\frac{1}{2} \frac{d}{dt} \int |\widehat{I w_n}|^2 \chi_n^{-1} d\xi & = 3 \int \q F(u_n^2(Iw_n)_x) \overline{\widehat{I w_n}} d\xi + 3 \int \q F(u_n^2(Ie^{-t\partial_x^3}z)_x) \overline{\widehat{I w_n}} d\xi\\
& \qquad - \Im \int \xi \frac{\chi_n'}{\chi_n} \q F(u_n^3-S_n^3) \overline{\widehat{I w_n}} d\xi.
\end{align*}
Now by Plancherel and integration by parts,
\begin{align*}
\MoveEqLeft \int \q F(u_n^2(Iw_n)_x) \overline{\widehat{I w_n}} d\xi  = \frac{1}{2\pi} \int u_n^2(Iw_n)_x (I w_n) dx = - \frac{1}{2\pi} \int (u_n)_x u_n (Iw_n)^2 dx \\
\MoveEqLeft \left| \int \q F(u_n^2(Iw_n)_x) \overline{\widehat{I w_n}} d\xi \right|  \lesssim \frac{\delta^2}{t} \| Iw_n \|_{L^2}^2 \lesssim  \frac{\delta^2}{t} \| \widehat{Iw_n} \|_{L^2}^2, \\
\MoveEqLeft \left| \int \q F(u_n^2(Ie^{-t\partial_x^3}z)_x)  \overline{\widehat{I w_n}} d\xi \right|  \le \| u_n^2(Ie^{-t\partial_x^3}z)_x \|_{L^2} \| \widehat{I w_n} \|_{L^2} \\
& \lesssim \| u_n \|_{L^\infty}^2 \| (Ie^{-t\partial_x^3}z)_x \|_{L^2} \| \widehat{I w_n} \|_{L^2} \\
& \lesssim \frac{\delta^2}{t^{2/3}} \| \partial_\xi \hat z \|_{L^2} \| \widehat{I w_n} \|_{L^2} \lesssim  \frac{\delta^3}{t^{2/3}} \| \widehat{I w_n} \|_{L^2}, \\
\MoveEqLeft \left|  \int \xi \frac{\chi_n'}{\chi_n} \q F(u_n^3-S_n^3) \overline{\widehat{I w_n}} d\xi \right| = 2 \left|  \int \xi (\chi_n^{1/2})' \q F(u_n^3-S_n^3) \overline{\widehat{I w_n}} \chi_n^{-1/2} d\xi \right|,  \\
& \le \| \xi (\chi_n^{1/2})' \|_{L^\infty} \|  u_n^3-S_n^3 \|_{L^2} \| \widehat{I w_n} \chi_n^{-1/2} \|_{L^2} \\
& \lesssim \left( \frac{A\delta^5}{t^{13/18}} + \frac{\delta^3}{t^{5/9}} \right)  \| \widehat{I w_n} \chi_n^{-1/2} \|_{L^2}.
\end{align*}
(We also used \eqref{est:boot_NL} for the last estimate).
As $0< \chi_n < 1$, $ \| \widehat{I w_n} \|_{L^2} \le \| \widehat{I w_n}\chi_n^{-1/2} \|_{L^2}$. Therefore, we obtain, for $t<\tau_n$,
\begin{align*}
\frac{1}{2}\frac{d}{dt}\left\| \widehat{Iw_n}(t) \chi_n^{-1/2}\right\|_{L^2}^2  & \lesssim \frac{\delta^2}{t}\left\| \widehat{Iw_n}(t) \chi_n^{-1/2}\right\|_{L^2}^2 + \left(\frac{\delta^3}{t^{2/3}}+ \left( \frac{A\delta^5}{t^{13/18}} + \frac{\delta^3}{t^{5/9}} \right) \right) \left\| \widehat{Iw_n}(t) \chi_n^{-1/2}\right\|_{L^2}. 
	\end{align*}
This implies that, for some universal constant $C$,
\[ \left| \frac{d}{dt} \left( t^{-2C\delta^2} \left\| \widehat{Iw_n}(t) \chi_n^{-1/2}\right\|_{L^2} \right) \right| \le 2C t^{- 2C\delta^2}  \left( \frac{\delta^3}{t^{2/3}} + \frac{A\delta^5}{t^{13/18}} \right). \]
Now, for $\delta < \delta_0$ small enough, $2C\delta^2 < 5/18$ so that the right hand side is integrable in time and due to \eqref{est:Iwn_init},  $t^{-2C\delta^2} \left\| \widehat{Iw_n}(t) \chi_n^{-1/2}\right\|_{L^2} \to 0$ as $t \to 0^+$. Hence, we can integrate the above estimate on $[0,t]$ and get
\begin{equation}
 \forall t \in [0,\tau_n), \quad  \left\| \widehat{Iw_n}(t) \chi_n^{-1/2}\right\|_{L^2} \lesssim \delta^3 t^{5/18} \left( t^{1/18} +  A\delta^2 \right),
\end{equation}
which is \eqref{lem:aprioriI}.

\bigskip

\emph{Step 3. A priori estimate for $\partial_t w_n$.} We claim that for all $n \in \m N$,
	\begin{align} \label{lem:aprioriLinfty}
\forall t \in (0,T_n), \quad \left\| \partial_t \tilde{w}_n(t) \chi_n^{-1} \right\|_{L^\infty} & \lesssim \frac{1}{t}\|w_n(t)\|_{\mathcal{E}(t)}\left(\|w_n(t)\|_{\mathcal{E}(t)}^2 + \|S\|_{\mathcal{E}(1)}^2\right) \\
	& \qquad + \frac{1}{t^{8/9}}\nz{z}\left(\|w_n(t)\|_{\mathcal{E}(t)}^2+ \|S\|_{\mathcal{E}(1)}^2 + \nz{z}^2\right). \nonumber
	\end{align}

Indeed, we have
\[ \partial_t \tilde{w}_n(t) \chi_n^{-1} = N[\tilde S_n + \tilde z + \tilde w_n](t) - N[\tilde S_n](t). \]
For the nonlinear terms with at least one $z$, we use Lemma \ref{lem:termos_com_z} which gives the pointwise (in $\xi$) bound
\[ \frac{1}{t^{8/9}} (\| S \|_{\q E(1)} + \| w_n(t) \|_{\q E(t)})^2 \nz{z} + \frac{1}{t^{2/3}} (\| S \|_{\q E(1)} + \| w_n(t) \|_{\q E(t)}) \nz{z}^2 + \nz{z}^3. \]
The remaining terms are $N[\tilde S_n, \tilde S_n, \tilde w_n]$, $N[\tilde S_n, \tilde w_n, \tilde w_n]$ and $N[\tilde w_n, \tilde w_n, \tilde w_n]$ (they all have at least one $w_n$). Using Lemma \ref{lem:N_E}, they are bounded pointwise by
\[ \frac{1}{t} (\| S \|_{\q E(1)} + \| w_n(t) \|_{\q E(t)})^2 \| w_n \|_{\q E(t)}. \]
This gives \eqref{lem:aprioriLinfty}. If we restrict to the interval $(0,\tau_n)$, using \eqref{est:boot_w_E}, and $A\delta^2 \le1$, this rewrites simply
\begin{equation} \label{est:boot_wt}
\forall t \in (0,t_n), \quad  \left\| \partial_t \tilde{w}_n(t) \chi_n^{-1} \right\|_{L^\infty} \lesssim \frac{\delta^3}{t^{8/9}}.
\end{equation}

\bigskip

\emph{Step 4. Bound on $f_n(t)$}.
Let us bound separately the two terms of $f_n$. First, after integration in time of \eqref{est:boot_wt} (recall that $\| \tilde w_n(t) \chi_n^{-1} \|_{L^\infty} \to 0$ as $t \to 0^+$), we get
\begin{align} \label{est:step2_a}
\forall t \in [0,\tau_n], \quad \|\tilde{w}_n (t) \chi_n^{-1}\|_{L^\infty} \lesssim \delta^3 t^{1/9}.
\end{align}
Second, in view of the definition of $I$, we can write
\[ \partial_\xi \tilde{w}_n = -ie^{it\xi^3}\widehat{Iw_n} +\frac{3t}{\xi}\partial_t \tilde{w}_n = -ie^{it\xi^3}\widehat{Iw_n} +3te^{-it\xi^3}\chi_n \q F(u_n^3-S_n^3), \]
and so, using \eqref{lem:aprioriI} and \eqref{est:boot_NL}, we infer
\begin{align}
\left\| \partial_\xi \tilde{w}_n \chi_n^{-1/2} \right\|_{L^2} & \le \left\| \widehat{Iw_n}(t) \chi_n^{-1/2}\right\|_{L^2} + 3t\|u_n^3-S_n^3\|_{L^2} \\
& \lesssim \delta^3 t^{5/18} (t^{1/18} + A\delta^2) + A\delta^5 t^{5/18} + \delta^3 t^{4/9}\\& \lesssim \delta^3t^{1/3}+A\delta^5t^{5/18}.\label{est:step2_b}
\end{align}
Estimates \eqref{est:step2_a} and \eqref{est:step2_b} give that, for all $t \in [0,\tau_n]$,
\begin{align*}
	f_n(t) & = t^{-1/9} \left( \|\tilde{w}_n (t) \chi_n^{-1}\|_{L^\infty} + t^{-1/6} \| \partial_\xi \tilde{w}_n \chi_n^{-1/2} \|_{L^2} \right) \\
	& \lesssim  t^{-1/9} \left( \delta^3 t^{1/9} + t^{-1/6} ( \delta^3 t^{1/3} + A \delta^5  t^{5/18}) \right) \lesssim \delta^3 + A \delta^5.
\end{align*}
Let us remark that both terms in $f_n$ contribute (the leading powers of $t$ cancel for both).

Since $A\delta^2 \le 1$, there exists a universal constant $M$ such that
\begin{equation}\label{eq:estimat_fn}
\forall t\in [0,\tau_n], \quad f_n(t)\le M \delta^3.
\end{equation}

\bigskip

\emph{Step 5. Closing the bootstrap.} Choosing
\[ \delta_0=\min\left\{ \frac{1}{2\sqrt{M}}, \sqrt{\frac{5}{36C}} \right\} \quad \text{and then} \quad  A=1/\delta_0^2, \] 
Step 2 and \eqref{eq:estimat_fn} imply that
\[ \forall t \in [0,\tau_n], \quad  f_n(t) \le \frac{A}{2} \delta^3. \]
From a continuity argument, we necessarily have $\tau_n = \min(1,T_n)$. But if $T_n \le 1$, then we also have 
\[ \sup_{t\in [0,T_n)} \| w_n(t) \|_{X_n(t)} \le \sup_{t \in [0,T_n)} f_n(t) \le A\delta^3, \]
which contradicts the maximality of $T_n$ in Proposition \ref{lem:existwn}. Hence $T_n > 1$ and $\tau_n =1$. 

As $\| w_n(t) \|_{\q E(t)} \le t^{1/9} f_n(t)$, this gives the first part of \eqref{est:wn_E_L2}. Also notice that $\| \partial_x \tilde w_n(t) \|_{L^2} \le t^{5/18} f_n(t)$, and in view of \eqref{est:boot_wt}, we also obtain both estimates of \eqref{est:wn_xL2}.

\bigskip

\emph{Step 6. $L^2$ bound on $w_n$.} Finally, we will prove that 
\begin{equation}\label{lem:L2}
\| \hat w_n(t) \chi_n^{-1/2} \|_{L^2} \lesssim  A \delta^6 t^{1/9},\quad \forall t\in[0,1].
\end{equation} 

	Notice that, for each $n$, $\| \hat w_n \chi^{-1} \|_{L^\infty} \lesssim A \delta^3 t^{1/9}$, so that one has for free that $w_n \chi^{-1/2} \in L^\infty([0,1], L^2)$ and for all $t \in [0,1]$
\begin{align} \label{est:L2_chi_n}
 \| \hat w_n(t) \chi_n^{-1/2} \|_{L^2} \lesssim \| \hat w_n(t) \chi^{-1} \|_{L^\infty} \| \chi_n^{1/2} \|_{L^2} \lesssim_n A \delta^3 t^{1/9}.
 \end{align}
Since
	\[
	\partial_t \hat w_n - i \xi^3 \hat w_n =  i \xi \chi_n  \q F(u_n^3 - S_n^3),
	\]
 multiplying by $\overline{\hat w_n} \chi_n^{-1}$, integrating in $\xi$ and taking the real part
\begin{align*}
\frac{1}{2}\frac{d}{dt} \int |\hat w_n(t)|^2 \chi_n^{-1} d\xi = \Re \int i\xi  \q F((u_n^3 - S_n^3) \overline{\hat w_n} d\xi = - \int  (u_n^3 - S_n^3) \partial_x w_n dx.
\end{align*}
We claim that
\begin{align} \label{est:NL_L2}
\MoveEqLeft \left|\int (u_n^3 - S_n^3) \partial_x w_n dx \right| \lesssim \frac{1}{t} \left( \| w_n(t) \|_{\q E(t)} + \| S_n \|_{\q E(1)}^2 + \| e^{-t\partial_x^3}z \|_{\q E(t)}^2 \right) \\
& \qquad \cdot \left( \| w_n \|_{L^2}^2 + \| z \|_{L^1} \| w_n(t)  \|_{\q E(t)} \right).
\end{align}
To see this, we expand the terms (made of 4 factors) and split them depending on whether $w_n$ occurs  at most once or at least twice. In the following discussion, the factors $v_1, v_2$ stand for either one of $w_n$, $S_n$ or $e^{-t\partial_x^3}z$.

For terms where $w_n$ occur at most once, $e^{-t\partial_x^3}z$ also appear, and they all can take the form
\[ \int v_1 v_2 (e^{-t\partial_x^3} z) \partial_x w_n = \int \q F(v_1 v_2  \partial_x w_n) \overline{\q F(e^{-t\partial_x^3} z)} d\xi, \]
which is bounded by
\begin{align*}
\| \q F(v_1 v_2  \partial_x w_n) \|_{L^\infty} \| \q F(e^{-t\partial_x^3} z ) \|_{L^1} & = \| N[\tilde v_1, \tilde v_2, i\xi \tilde w_n](t) \|_{L^\infty} \| \hat z \|_{L^1} \\
& \lesssim \frac{1}{t} \| v_1(t) \|_{\q E(t)} \| v_2(t) \|_{\q E(t)} \| w_n(t) \|_{\q E(t)} \| \hat z \|_{L^1},
\end{align*}
where we used Lemma \ref{lem:estL2} in a crucial way.

For terms where $w_n$ occur at least twice, by integration by parts, we see that they can all take the form
\[ \int \partial_x(v_1 v_2) (w_n)^2 dx \]
and so, they are bounded by
\[ \| \partial_x(v_1 v_2)(t) \|_{L^\infty} \| w_n(t) \|_{L^2}^2 \lesssim \frac{1}{t} \| v_1(t) \|_{\q E(t)} \| v_2(t) \|_{\q E(t)} \| w_n(t) \|_{L^2}^2. \]
This proves \eqref{est:NL_L2}. 

Since $\| w_n \|_{L^2} \lesssim  \| \hat w_n \chi_n^{-1/2} \|_{L^2}$, and, by Step 5, $\| w_n(t) \|_{\q E(t)} \lesssim A \delta^3 t^{1/9}$ (valid for $t \in [0,1]$), \eqref{est:NL_L2} implies that
\[ \left| \frac{d}{dt} \| \hat w_n(t) \chi_n^{-1/2} \|_{L^2}^2 \right| \lesssim \frac{\delta^2}{t} \left( \| \hat w_n(t) \chi_n^{-1/2} \|_{L^2}^2 + A \delta^4 t^{1/9} \right). \]
Recalling that from \eqref{est:L2_chi_n}, $\| \hat w_n(t) \chi_n^{-1/2} \|_{L^2}  = O(t^{1/9})$, a Gronwall argument as in Step 2 gives \eqref{lem:L2}.
By Plancherel, the $L^2$ bound in \eqref{est:wn_E_L2} follows.
\end{proof}

\section{Proof of the main result}\label{sec:proofthm}
\begin{prop}[Uniqueness]\label{prop:uniq}
	There exists a universal constant $K>0$ such that, given $\eta>\delta^2/K^2$ and $T>0$, there exists at most one solution to \eqref{eq:w} satisfying
	\[
	\forall t \in (0,T], \quad \|w(t)\|_{\mathcal{E}(t)}\le K\sqrt{\eta} \quad \text{and} \quad  \|w(t)\|_{L^2}\le t^{\eta}.
	\]
\end{prop}

\begin{proof}
	Suppose $w_1, w_2$ are two solutions in the above conditions and set $w=w_1-w_2$. Then
	\[
	\partial_t w + \partial_x^3 w = (u_1^3-u_2^3)_x, \quad u_j(t)=S(t)+ e^{-t\partial_x^3}z + w_j(t), \ j=1,2.
	\]
	Observe that
	\[
	\|u_1(t)\|_{\mathcal{E}(t)},\ \|u_2(t)\|_{\mathcal{E}(t)}\lesssim \delta \lesssim K\sqrt{\eta}.
	\]
	Define $A=\sup_{t\in [0,T]} t^{-\eta}\|w(t)\|_{L^2}$ (which is finite by assumption). Direct integration and Lemma \ref{lem:estE} give that for all $t \in [0,T]$,
	\begin{align*}
	\|w(t)\|_{L^2}^2 & \le \int_0^t \int (u_1^3-u_2^3)_x w dx ds \le \frac{1}{2} \int_0^t \|(u_1^2 + u_1u_2+u_2^2)_x\|_{L^\infty}\|w(s)\|_{L^2}^2ds \\
	& \le C\int_0^t \frac{K^2\eta}{s}As^{2\eta}ds \le \frac{CK^2}{2} A^2t^{2\eta}.
	\end{align*}
	Dividing by $t^{2\eta}$ and taking the supremum in $t \in [0,T]$, we get $A^2\le \frac{CK^2}{2}A^2$, which implies $A=0$ for $K^2<2/C$.
\end{proof}

\begin{nb}
	In \cite{CCV20}, forward uniqueness of solutions in $\mathcal{E}$ was obtained for strictly positive times. The argument goes through an estimate for the $L^2$ norm on positive half-lines (which is finite for elements in $\mathcal{E}$).The bounds given by Lemma \ref{lem:estE} give a behavior of $1/t$, which must then be integrated in $(0,T)$. This can be compensated if one assumes the polynomial growth $\|w(t)\|_{L^2_x}\le t^\eta$. In conclusion, this strategy can be used to provide an alternate proof of Proposition \ref{prop:uniq} but not to further improve it.
\end{nb}

\begin{proof}[Proof of Theorem \ref{teo:existw}]
	Consider the approximations $w_n$ defined in Proposition \ref{lem:existwn}. By Proposition \ref{prop:aprioriI}, these solutions are defined on $[0,1]$ and
	\begin{equation}\label{eq:unifbound}
	\|w_n(t)\|_{\mathcal{E}(t)} \lesssim \delta^3 t^{1/9}, \quad \|w_n(t)\|_{L^2}\lesssim \delta^3 t^{1/18}.
	\end{equation}
	Notice that we also have, for $t\in(0,1]$,
	\[ \| \partial_\xi \tilde w_n(t)\|_{L^2} \lesssim \delta^3 t^{5/18} \quad \text{and} \quad \| \partial_t \tilde w_n(t)\|_{L^\infty} \lesssim \frac{\delta^3}{t^{8/9}}, \]
	so that, by Sobolev embedding,
	\begin{equation}\label{eq:holdercontxi}
	\forall t \in [0,1], \ \forall \xi_1,\xi_2 \in \m R, \quad |\tilde{w}_n(t,\xi_1)-\tilde{w}_n(t,\xi_2)|\lesssim \delta^3 t^{5/18}|\xi_1-\xi_2|^{1/2},
	\end{equation}
	and,
	\begin{equation}\label{eq:holdercontt}
	\forall t_1, t_2 \in [0,1], \ \forall \xi \in \m R, \quad |\tilde{w}_n(t_1,\xi)-\tilde{w}_n(t_2,\xi)|\lesssim \delta^3 |t_1^{1/9}-t_2^{1/9}|.
	\end{equation}
	Consequently, for any $R>0$, $(\tilde{w}_n)_{n\in\mathbb{N}}$ is equibounded and equicontinuous on $[0,1]\times [-R,R]$. By Ascoli-Àrzela theorem,
	\[
	\tilde{w}_n \to \tilde{w}\quad \text{uniformly in }[0,T]\times[-R,R]
	\] 
	and $\tilde{w}$ satisfies \eqref{eq:unifbound}, \eqref{eq:holdercontxi} and \eqref{eq:holdercontt}. In particular, $w\in \mathcal{E}((0,1)) \cap L^\infty((0,1),L^2(\R))$ and bound \eqref{eq:estfinal} holds.
	 
	We now prove that $w$ solves \eqref{eq:w} in the sense of distributions. Since $w_n$ is uniformly bounded in $\mathcal{E}((0,1))$, Lemma \ref{lem:estE} implies that $w_n$ is also equibounded and equicontinuous on $[\varepsilon, 1]\times [-R,R]$ for any $\varepsilon>0$ and $R>0$. Thus
	\[
	w_n\to w \quad \text{uniformly in }[\varepsilon,1]\times[-R,R].
	\]
	Since
	\[
	|w_n(t,x)|\lesssim \frac{1}{t^{1/3}\jap{x/t^{1/3}}^{1/4}},
	\]
	the uniform convergence implies that $w_n\to w$ in $L^\infty((\varepsilon,1), L^6(\R))$. The exact same reasoning also yields $S_n\to S$ in $L^\infty((\varepsilon,1), L^6(\R))$. These convergences can now be used to conclude that
	\[
	(u_n^3-S_n^3)_x\to (u^3-S^3)_x\text{ in }\mathcal{D}'((\varepsilon, 1) \times \R)
	\]
	and that $w$ satisfies \eqref{eq:w} in the distributional sense on $(0,1)$. 
	
	To extend the solution up to $t=+\infty$, observe that
	\[
	\forall t \in (0,1], \quad \|u(t)\|_{\mathcal{E}(t)}\le \|S(t)\|_{\mathcal{E}(t)} + \|e^{-t\partial_x^3}z\|_{\mathcal{E}(t)} + C \delta^3 t^{1/9}\le 3\delta.
	\] 
	Thus the global existence result of \cite[Theorem 2]{CCV20} can be applied (at $t=1/2$) to extend $u$ for all positive times. Finally, Proposition \ref{prop:uniq} with $\eta=1/18$ gives the uniqueness property (decreasing the value of $\delta_0$ further, if necessary).
\end{proof}

\bibliography{biblio_mkdv}
\bibliographystyle{plain}

\bigskip

\normalsize

\begin{center}
	{\scshape Simão Correia}\\
{\footnotesize
	Center for Mathematical Analysis, Geometry and Dynamical Systems,\\
	Department of Mathematics,\\
	Instituto Superior T\'ecnico, Universidade de Lisboa\\
	Av. Rovisco Pais, 1049-001 Lisboa, Portugal\\
	simao.f.correia@tecnico.ulisboa.pt
}
	\bigskip
	
	{\scshape Raphaël Côte}\\
	{\footnotesize
		Université de Strasbourg\\
		CNRS, IRMA UMR 7501\\
		F-67000 Strasbourg, France\\
		{cote@math.unistra.fr}
	}

\end{center}

\end{document}